\documentclass[11pt,a4paper]{article}
\usepackage[full]{textcomp}
\usepackage{newtxtext} 
\usepackage[utf8]{inputenc}

\usepackage{graphicx,amsmath,amssymb} 
\usepackage{mathtools}
\usepackage{enumitem}
\setlist{topsep=1pt,parsep=1pt,itemsep=1pt} 
\usepackage{tikz}
\usetikzlibrary{matrix,arrows}
\usetikzlibrary{decorations.pathmorphing}
\usetikzlibrary{decorations.markings}
\usepgflibrary{fpu}
\usepackage{xcolor}
\usepackage{mathrsfs} 
\usepackage{bbm}

\usepackage{cleveref}
\usepackage{amsthm}  
\usepackage{autonum} 

\usepackage{thmtools}

\input{sr_spde.sty}
\numberwithin{equation}{section}

\title{Stochastic resonance in stochastic PDEs}
\author{Nils Berglund and Rita Nader}
\date{}

\begin{document}

\maketitle

\begin{abstract}
We consider stochastic partial differential equations (SPDEs) on the 
one-dimensional torus, driven by space-time white noise, and with a 
time-periodic drift term, which vanishes on two stable and one unstable 
equilibrium branches. Each of the stable branches approaches the unstable one 
once per period. We prove that there exists a critical noise intensity, 
depending on the forcing period and on the minimal distance between equilibrium 
branches, such that the probability that solutions of the SPDE make transitions 
between stable equilibria is exponentially small for subcritical noise 
intensity, while they happen with probability exponentially close to $1$ for 
supercritical noise intensity. Concentration estimates of solutions are given 
in the $H^s$ Sobolev norm for any $s<\frac12$. The results generalise to an 
infinite-dimensional setting those obtained for $1$-dimensional SDEs 
in~\cite{BG_SR}. 
\end{abstract}

\leftline{\small{\it Date.\/} July 15, 2021. 
Updated August 9, 2021.
}
\leftline{\small 2020 {\it Mathematical Subject Classification.\/} 
60H15, 		
60G17		
(primary),
34F15,   	
37H20   	
(secondary)}


\section{Introduction}
\label{sec:intro} 

Stochastic resonance can occur when a bistable or multistable dynamical system 
is forced periodically in time, while also subjected to noise. When the forcing 
period is close to the typical time needed by the noise to move the system from 
one metastable state to another one, large-amplitude, nearly periodic 
oscillations may occur. Even if this resonance condition is not exactly met, 
the response of the system shows a trace of the periodic forcing in its power 
spectrum.

The mechanism of stochastic resonance was initially introduced in the context of 
climate science~\cite{Nicolis,BSV}, to propose an explanation for the relation 
between Milankovitch cycles and glacial periods. Since then, stochastic 
resonance has shown up in several other applications to ecology and climate 
science, see for instance~\cite{Velez01,Eisenman_Wettlaufer_08,ATW_21}. It also 
appears in many other applications, including 
neuroscience~\cite{MuratovVanden-EijndenE} and quantum electronics~\cite{WM}. We 
refer to~\cite{WJ,GHM,HanggiReview02} for comprehensive reviews on this topic.

The most precise mathematical results on stochastic resonance have been 
obtained for one-dimensional stochastic differential equations (SDEs) of the 
form 
\begin{equation}
\label{eq:SDE} 
 \6X_t = f(\eps t,X_t) \6t + \sigma \6W_t\;,
\end{equation} 
where $W_t$ is a standard Wiener process, and $f$ is a time-periodic bistable 
drift term. A standard example is 
\begin{equation}
\label{eq:f_AC} 
 f(\eps t,x) = x - x^3 + A\cos(\eps t)
 = - \frac{\partial}{\partial x} \biggpar{\frac14 x^4 - \frac12 x^2 - 
A\cos(\eps t) x}\;.
\end{equation} 
Whenever $A$ is smaller than a critical value given by $\Acrit = 
\frac{2}{3\sqrt{3}}$, the drift term vanishes in three different values of $x$, 
which correspond to equilibrium states of the system with a frozen value of 
$\eps t$. These states are also critical points of the double-well potential 
$V(x,\eps t) = \frac14 x^4 - \frac12 x^2 - A\cos(\eps t) x$, where the middle 
point is the unstable saddle, and the two outer points are stable potential 
minima. 

The first investigations of stochastic resonance in systems of the 
form~\eqref{eq:SDE} focused on the case of small amplitude 
$A$~\cite{Fox,GMSMP,JH2}, but many other parameter regimes have been considered 
as well (see~\cite{HIPP_book} for an overview of mathematical results). Here, 
we will be mainly interested in the case where $A$ is slightly smaller than 
$\Acrit$, which was analysed in the one-dimensional setting in the 
work~\cite{BG_SR}. In that situation, one can prove that there exists a critical 
noise intensity $\sigmac$ such that when $\sigma \ll \sigmac$, transitions 
between potential minima are very rare, while for $\sigma \gg \sigmac$, it is 
very likely that the system goes back and forth between the local minima twice 
per period.

The present work is concerned with a generalisation of~\eqref{eq:SDE} to the 
infinite-dimensional setting. We will consider stochastic partial differential 
equations (SPDEs) of the form 
\begin{equation}
\label{eq:SPDE_intro} 
\6\phi(t,x) 
= \bigbrak{\Delta\phi(t,x) + f(\eps t,\phi(t,x))} \6t 
+ \sigma\6W(t,x)\;,
\end{equation} 
where $x$ belongs to the one-dimensional torus $\T = \R/L\Z$, and
$W(t,x)$ denotes space-time white noise given by a cylindrical Wiener 
process. The drift term $f$ is again assumed to describe a bistable situation. 
For instance, the choice~\eqref{eq:f_AC} corresponds to a periodically forced 
Allen--Cahn equation. Our results apply, however, to more general drift terms 
$f$, that only need to satisfy a number of regularity and growth conditions. 

The analysis requires an extension to the infinite-dimensional situation of 
SPDEs of sample-path methods introduced in~\cite{BG_pitchfork,BG_SR} for the 
one-dimensional setting, and extended 
in~\cite{berglund2002GeoPerSDE,Berglund_Gentz_book} to arbitrary finite 
dimensions. A first step towards extending those methods to  
infinite dimensions has been taken in~\cite{gnann_kuehn_pein_2019}. 
However, that work considers noise that is coloured in space and white in time, 
given by a $Q$-Wiener process with trace class covariance, while we consider 
here the more difficult situation of space-time white noise. 

Our main results can be summarised as follows. As above, we assume that the 
time-periodic drift term $f$ vanishes on three branches, two of which come close 
to each other or meet once per period. The minimal distance between the branches 
at these close encounters is measured by a small parameter $\delta$, which 
corresponds to $\Acrit - A$ in the particular case where $f$ is given 
by~\eqref{eq:f_AC}. We then have the following results.

\begin{itemize}
\item 	Theorem~\ref{thm:stable} states that as long as the equilibrium branches 
are well-separated, solutions of the SPDE~\eqref{eq:SPDE_intro} are likely to 
remain close to deterministic solutions tracking the stable branches. Closeness 
is measured in the $H^s$ Sobolev norm, where $s$ is strictly smaller than 
$\frac12$, but can be arbitrarily close to $\frac12$.  

\item 	When equilibrium branches become close to each other, we decompose the 
solution $\phi(t,x)$ into its spatial mean $\phi_0(t)$, and its zero-mean 
transverse part $\phi_\perp(t,x)$. Theorem~\ref{thm:phiperp} says that the 
conclusion of Theorem~\ref{thm:stable} remains valid at bifurcation points for 
the transverse part. 

\item 	The behaviour of the spatial mean $\phi_0(t)$ depends on the value of 
the noise intensity $\sigma$. Theorem~\ref{thm:phi0_stable} implies that in the 
weak-noise regime $\sigma \ll \sigmac = (\delta\vee\eps)^{3/4}$, sample paths 
are still likely to remain close to the same stable equilibrium. The probability 
of making a transition to the other stable equilibrium is exponentially small 
in $\sigmac^2/\sigma^2$. 

\item 	In the strong-noise regime $\sigma \geqs \sigmac = 
(\delta\vee\eps)^{3/4}$, transitions between equilibrium branches become more 
likely. Theorem~\ref{thm:phi0_strong} implies that the probability not to make 
a transition to the other stable equilibrium when approaching an avoided 
bifurcation point decays roughly like 
$\exp[-\sigma^{4/3}/(\eps\log(\sigma^{-1}))]$. 
\end{itemize}

Our results thus show that similarly to the one-dimensional situation 
considered in~\cite{BG_SR}, depending on the noise intensity, transitions 
between stable equilibria are either exponentially rare, or happen with a 
probability exponentially close to $1$. There are some differences in the error 
terms, which are due to the fact that we have to deal with the transverse part 
$\phi_\perp$ of the solution. 

The main difficulty of the analysis comes from the fact that we work with 
space-time white noise in an infinite-dimensional situation. This prevents us 
from applying directly the methods from~\cite{berglund2002GeoPerSDE}, which 
work in finite dimension, and include dimension-dependent error terms. These 
error estimates can be adapted to trace class noise, as was done 
in~\cite{gnann_kuehn_pein_2019}, but the white noise case needs a different 
approach, relying on more careful estimates in various Sobolev norms. Key 
results are an estimate for a linearised equation based on the Fourier 
decomposition, presented in Section~\ref{ssec:stable_stochastic}, and a 
Schauder estimate given in Lemma~\ref{lemma3.5}. 

The remainder of this paper is organised as follows. In 
Section~\ref{sec:results}, we give the precise assumptions on the SPDEs we 
consider, and all main results, as well as a discussion of the different 
parameter regimes. Section~\ref{sec:proof_stable} contains the proofs for the 
stable case, that is, as long as the system does not approach any bifurcation 
points, while Section~\ref{sec:proof_bif} contains the proofs for the cases 
with (avoided) bifurcations. Appendix~\ref{annexes} recalls several 
inequalities involving products in Sobolev spaces that are used in the 
analysis. 

\subsection*{Notations}

The system studied in this work depends on three small parameters $\eps$,  
$\sigma$ and $\delta$. We write $X \lesssim Y$ to indicate that $X \leqs c Y$ 
for a constant $c$ independent of $\eps$,  $\sigma$ and $\delta$, as long as 
these parameters are small enough. The notation $X\asymp Y$ indicates that one 
has both $X \lesssim Y$ and $Y\lesssim X$, while Landau's notation $X = 
\Order{Y}$ means that $\abs{X} \lesssim Y$. If $a, b\in\R$, $a\wedge b$ denotes 
the minimum of $a$ and $b$, and $a\vee b$ denotes the maximum of $a$ and $b$. 
Finally, we write $1_{\cD}(x)$ for the indicator function of a set or event 
$\cD$. 

\subsection*{Acknowledgments}

This work is supported by the ANR project PERISTOCH, ANR–19–CE40–0023. The 
authors thank G\'erard Bourdaud for drawing their attention to the 
reference~\cite{Bourdaud_calcul_symbolique}.


\section{Main results}
\label{sec:results} 


\subsection{The set-up}
\label{ssec:setup} 

Let $L, T>0$ be real parameters. We will consider time-dependent SPDEs on the 
torus $\T = \R/L\Z$ of the form 
\begin{equation}
\label{eq:SPDE_fast}
\6\phi(t,x) 
= \bigbrak{\Delta\phi(t,x) + f(\eps t,\phi(t,x))} \6t 
+ \sigma\6W(t,x)\;,
\end{equation} 
for the unknown $\phi: I\times\T \to \R$, where $I = [0,T]$. Here

\begin{itemize}
 \item 	$\eps>0$ is a small parameter quantifying the slow time dependence;
 \item 	$\sigma>0$ is a small parameter measuring the noise intensity;
 \item 	$f: [0,T]\times\R \to \R$ is a forcing term satisfying a number of 
assumptions given below;
 \item $\6W(t,x)$ denotes space-time white noise on $\R_+\times\T$.
\end{itemize}
Our results extend naturally to the case where $f : \R\times\R \to \R$ is 
periodic in the time variable, with period $T$. 

It will be more convenient to work with slow time $\eps t$. Scaling time by a 
factor $\eps$ yields the equation 
\begin{equation}
\label{eq:SPDE}
\6\phi(t,x) 
= \frac{1}{\eps} \bigbrak{\Delta\phi(t,x) + f(t,\phi(t,x))} \6t 
+ \frac{\sigma}{\sqrt{\eps}} \6W(t,x)\;.
\end{equation} 
It will sometimes be useful to work with a potential $U$ associated with $f$, 
satisfying 
\begin{equation}
\label{eq:def_U} 
 f(t,\phi) = -\partial_\phi U(t,\phi)\;.
\end{equation} 
The following assumption on the behaviour of $U$ for large values of $\phi$ 
will be assumed to hold throughout this work.

\begin{assumption}[Global behaviour of the drift term]
\label{assum:U_large_phi}
The potential $U$ admits, for all $(t,\phi) \in I\times \R$, a decomposition 
\begin{equation}
 U(t,\phi) = P(t,\phi) + g(t,\phi)
\end{equation} 
into a polynomial part and a bounded part. More precisely, 
\begin{itemize}
\item 	there exists an integer $p_0 \geqs 1$ such that the map $\phi\mapsto 
P(t,\phi)$ is a polynomial of degree $2p_0$, of the form 
\begin{equation}
 P(t,\phi) = \sum_{j=0}^{2p_0} A_j(t)\phi^j
\end{equation} 
with coefficients $A_j\in\cC^1(I,\R)$ such that $\abs{A_j(t)}$ and 
$\abs{A_j'(t)}$ are bounded uniformly, and $A_{2p_0}(t) > 0$ for all $t\in I$; 

\item 	the function $g\in\cC^2(I\times\R,\R)$ satisfies 
\begin{equation}
 \abs{g(t,\phi)\phi^{-1}}\;,\;
 \abs{\partial_\phi g(t,\phi)}\;,\;
 \abs{\partial_{\phi\phi} g(t,\phi)}\;,\;
 \abs{\partial_t g(t,\phi)} \leqs M 
\end{equation} 
for all $(t,\phi)\in I\times\R$ and some constant $M>0$. 
\end{itemize}
\end{assumption}


\subsection{The stable case}
\label{ssec:results_stable} 

We start by considering the case where the drift term $f$ admits a stable 
equilibrium branch, in the following sense. 

\begin{assumption}[Stable case]
\label{assum:a(t)}
There exists a map $\phi^*: I\to\R$ such that 
\begin{equation}
 f(t,\phi^*(t)) = 0 \qquad \forall t\in I\;.
\end{equation} 
Furthermore, the linearisation $a(t) = \partial_\phi f(t,\phi^*(t))$ satisfies 
\begin{equation}
 -a_+ \leqs a(t) \leqs -a_- 
 \qquad \forall t\in I
\end{equation} 
for some constant $a_\pm > 0$. 
\end{assumption}

Consider first the deterministic equation 
\begin{equation}
\label{eq:SPDE_det}
\6\phi(t,x) 
= \frac{1}{\eps} \bigbrak{\Delta\phi(t,x) + f(t,\phi(t,x))} \6t\;.
\end{equation} 
It will be convenient to work with the orthonormal Fourier basis 
$\set{e_k}_{k\in\Z}$ of $L^2(\T,\R)$, given by 
\begin{equation}
 e_k(x) = 
\begin{cases}
\displaystyle
 \sqrt{\frac{2}{L}} \cos \biggpar{\frac{k\pi x}{L}} 
 & \text{if $k > 0$\;,}\\[10pt]
\displaystyle
 \frac{1}{\sqrt{L}}  & \text{if $k = 0$\;,}\\[10pt]
\displaystyle
 \sqrt{\frac{2}{L}} \sin \biggpar{\frac{k\pi x}{L}} & \text{if $k < 0$\;.}
\end{cases}
\label{eq:Fourier} 
\end{equation} 
Given a real number $s>0$ and a function $\phi\in L^2(\T)$ with Fourier 
expansion 
\begin{equation}
 \phi(x) = \sum_{k\in\Z} \phi_k e_k(x)\;,
\end{equation} 
we define the fractional Sobolev norm of $\phi$ by 
\begin{equation}
\label{eq:def_Sobolev_norm} 
 \norm{\phi}_{H^s}^2 = \sum_{k\in\Z} \jbrack{k}^{2s} \phi_k^2\;,
\end{equation} 
where we use the \lq\lq Japanese bracket\rq\rq\ notation 
$\jbrack{k} = (1+k^2)^{1/2}$. We denote by $H^s = H^s(\T,\R)$ the 
fractional Sobolev space (or Bessel potential space) of functions 
$\phi:\T\to\R$ admitting a finite $H^s$-norm. 
We then have the following result, which generalises to our 
infinite-dimensional setting results from singular perturbation theory that are 
well-known in finite dimension (see in particular~\cite{Tihonov,Fenichel}). 

\begin{prop}[Deterministic dynamics in the stable case]
\label{det_prop}
There exist constants $C, \eps_0>0$ such that for $0 < \eps < \eps_0$, the 
equation~\eqref{eq:SPDE_det} admits a particular solution $\bar\phi(t,x)$ 
satisfying 
\begin{equation}
 \norm{\bar\phi(t,\cdot) - \phi^*(t) e_0}_{H^1} 
 \leqs C\eps
 \qquad \forall t\in I\;.
\end{equation} 
\end{prop}

In the finite-dimensional case, it is known 
(see~\cite[Theorem~2.4]{BG_pitchfork}) that solutions of the stochastic 
equation~\eqref{eq:SPDE}, starting near the equilibrium branch $\phi^*$, remain 
close to that branch with high probability. To quantify this in our 
infinite-dimensional situation, given $s>0$ we define for any $h > 0$ the set 
\begin{equation}
 \cB(h) 
 = \Bigset{(t,\phi) \colon t\in I, \norm{\phi - \bar\phi(t,\cdot)}_{H^s} < h}\;.
\end{equation} 
Given an initial condition $(0,\phi_0)$ in $\cB(h)$, the first-exit time from 
$\cB(h)$ is the stopping time 
\begin{align}
 \tau_{\cB(h)} 
 &= \inf \bigset{t>0 \colon (t,\phi(t,\cdot))\notin \cB(h)} \\
 &= \inf \bigset{t>0 \colon \norm{\phi - \bar\phi(t,\cdot)}_{H^s} \geqs h}\;.
\end{align} 
By convention, we set $\tau_{\cB(h)} = +\infty$ whenever 
$(t,\phi(t,\cdot))\in\cB(h)$ for all $t\in I$. 

\begin{theorem}[Stochastic dynamics in the stable case]
\label{thm:stable} 
For any $s\in(0,\frac12)$ and any $\nu > 0$, there 
exist constants $\kappa = \kappa(s), \eps_0, h_0$ and $C(\kappa,t,\eps,s) > 0$ 
such that, whenever $0 < \eps \leqs \eps_0$ and $0 < h \leqs h_0 \eps^{\nu}$, 
the solution of~\eqref{eq:SPDE} with initial condition $\phi(0,\cdot) = 
\bar\phi(0,\cdot)$ satisfies 
\begin{equation}
\bigprob{\tau_{\cB(h)} < t}
\leqs C(\kappa,t,\eps,s) \exp\biggset{-\kappa\frac{h^2}{\sigma^2} 
\biggbrak{1-\biggOrder{\frac{h}{\eps^\nu}}}}\;.
\end{equation}
for all $t\in I$. 
\end{theorem}

\begin{remark}
The proof yields explicit bounds on $C(\kappa,t,\eps,s)$. In particular, this 
quantity can be taken proportional to $t/\eps$, while its dependence on 
$\kappa$ and $s$ is more complicated. 
\end{remark}

\begin{remark}
The result also holds for general initial conditions $\phi(0,\cdot)$ in 
an $H^s$-neighbourhood of order $1$ of $\bar\phi(0,\cdot)$, provided one only 
considers the probability of leaving $\cB(h)$ after a time of order 
$\eps\log(\norm{\phi(0,\cdot)}_{H^s}h^{-1})$, since solutions need a time of 
that order to reach $\cB(h)$. See~\cite[Theorem~5.1.6]{Berglund_Gentz_book} for 
a precise formulation, which can be adapted to the present situation by a 
similar argument. 
\end{remark}


\subsection{Bifurcations and avoided bifurcations}
\label{ssec:results_bif} 

We now proceed to stating the main part of our results, which deal with systems 
admitting bifurcations or avoided bifurcations. As a motivating example, 
consider again the periodically forced Allen--Cahn equation 
\begin{equation}
 \6\phi(t,x) 
= \frac{1}{\eps} \bigbrak{\Delta\phi(t,x) + \phi(t,x) - \phi(t,x)^3 + A 
\cos(t)} \6t + \frac{\sigma}{\sqrt{\eps}} \6W(t,x)\;.
\end{equation}
Whenever $A < A_{\text{c}} = \frac{2}{3\sqrt{3}}$, the equation $\phi - \phi^3 
+ A\cos(t) = 0$ has exactly three solutions 
\begin{equation}
 \phi^*_1(t) < \phi^*_2(t) < \phi^*_3(t)\;.
\end{equation} 
If $t$ is replaced by a fixed parameter $t_0$, the equilibrium branches 
$\phi^*_{1,3}(t_0)$ are stable for the deterministic fast system 
\begin{equation}
\partial_t\phi(t,x) = \Delta\phi(t,x) + \phi(t,x) - \phi(t,x)^3 + A\cos(t_0)\;,
\end{equation} 
while $\phi^*_2(t_0)$ is unstable. If $A = A_{\text{c}}$, a stable branch and 
the unstable branch meet a transcritical bifurcation point whenever $t$ is a 
multiple of $\pi$. If $A$ is slightly smaller than $A_{\text{c}}$, the branches 
approach each other without quite touching. However, noise may trigger 
transitions between the branches, which is one of the basic mechanisms 
responsible for stochastic resonance. 

We will consider more general equations of the form~\eqref{eq:SPDE}, assuming 
that the drift term $f(t,\phi)$ vanishes on three equilibrium branches, two of 
which come close to each other at particular times. Whenever the three branches 
are well-separated, the dynamics near stable branches can be described by 
Theorem~\ref{thm:stable}. It is thus sufficient to describe the dynamics near 
times of bifurcation, or avoided bifurcation. By an affine change of variables, 
it is always possible to translate these (avoided) bifurcation points to the 
origin $(t,\phi) = (0,0)$. We will then make the following assumptions.

\begin{assumption}[Bifurcation point]
\label{assum:bifurcation} 
The drift term $f$ is of class $\cC^3$, and satisfies 
\begin{align}
f(t,0) &= \delta + a_1 t^2 + \Order{t^3}\;, \\
\partial_\phi f(t,0) &= \Order{t^2} \;, \\
\partial_{\phi\phi} f(0,0) &< 0
\end{align}
for constants $\delta \geqs 0$ and $a_1 > 0$.
\end{assumption}

Scaling time, space and $\phi$ appropriately, one can always assume that $a_1 = 
1$ (see Section~\ref{sec:proof_bif}). Under Assumption~\ref{assum:bifurcation}, 
one can check (see~\cite[Section~4]{BG_SR}) that in a neighbourhood of $(0,0)$, 
the drift term $f(t,\phi)$ vanishes only on two branches $\phi^*_\pm(t)$, 
satisfying 
\begin{align}
 \phi^*_\pm(t) & \asymp \pm \bigpar{\sqrt{\delta} + \abs{t}} \;, \\
 a_\pm(t) = \partial_\phi f(t,\phi^*_\pm(t)) 
 & \asymp \mp \bigpar{\sqrt{\delta} + \abs{t}}\;.  
\end{align}
In particular, $\phi^*_+$ is stable, while $\phi^*_-$ is unstable, unless 
$\delta = 0$ and $t=0$, when there is a transcritical bifurcation 
(\figref{fig:det_bif}). 

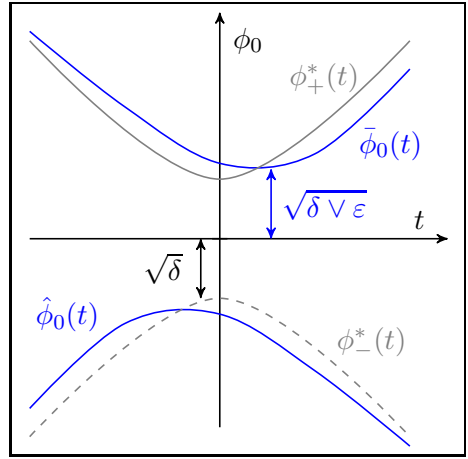
\begin{figure}
\begin{center}
\fbox{
\begin{tikzpicture}
[>=stealth',point/.style={circle,inner sep=0.035cm,fill=white,draw},
encircle/.style={circle,inner sep=0.07cm,draw},
x=2.5cm,y=2.5cm,declare function={f(\x) = sqrt((\x)^2+0.8)-0.01;
g(\x) = 0.2 -\x - sqrt(\x^2 + 0.1);}]

\draw[semithick] (-0.04,0) -- (0.04,0);
\draw[->,semithick] (0,-1) -> (0,1.2);
\draw[->,semithick] (-1,0) -> (1.2,0);

\draw[<->,semithick] (-0.1,-0.316) -- (-0.1,0);
\draw[<->,semithick,blue] (0.27,0) -- (0.27,0.37);


\draw[semithick,blue] plot[smooth,tension=.6]
  coordinates{(-1,1.1) (-0.5,0.7) (0,0.4) (0.5,0.45) (1,0.9)};

\draw[semithick,blue] plot[smooth,tension=.6]
  coordinates{(1,-1.1) (0.5,-0.7) (0,-0.4) (-0.5,-0.45) (-1,-0.9)};

  
\draw[semithick,gray] plot[smooth,tension=.6]
  coordinates{(-1,1.05) (0,0.316) (1,1.05)};

\draw[semithick,gray,dashed] plot[smooth,tension=.6]
  coordinates{(-1,-1.05) (0,-0.316) (1,-1.05)};



\node[gray] at (0.55,0.85) {$\phi^*_+(t)$};
\node[gray] at (0.8,-0.55) {$\phi^*_-(t)$};

\node[blue] at (0.9,0.5) {$\bar \phi_0(t)$};
\node[blue] at (-0.8,-0.4) {$\hat \phi_0(t)$};

\node[] at (1.05,0.1) {$t$};
\node[] at (0.15,1.05) {$\phi_0$};

\node[] at (-0.3,-0.158) {$\sqrt{\delta}$};
\node[blue] at (0.55,0.18) {$\sqrt{\delta\vee\eps}$};

\end{tikzpicture}
}
\end{center}
\vspace{-2mm}
\caption{Equilibrium branches and associated adiabatic solutions near the 
avoided bifurcation point $(0,0)$.}
\label{fig:det_bif} 
\end{figure}

In what follows, we will rewrite the SPDE~\eqref{eq:SPDE} in the form 
\begin{equation}
\label{eq:SPDE_bif} 
 \6\phi(t,x) 
 = \frac{1}{\eps} \Bigbrak{\Delta\phi(t,x) + g(t) - \phi(t,x)^2 
 - b(t,\phi(t,x))} \6t + \frac{\sigma}{\sqrt{\eps}} \6W(t,x)\;,
\end{equation} 
where 
\begin{align}
g(t) &= \delta + t^2 + \Order{t^3}\;, \\
b(t,\phi) &= \Order{\phi^3} + \Order{t\phi^2} + \Order{t^2\phi}\;.
\end{align}
It will be convenient to decompose the solution of~\eqref{eq:SPDE_bif} into its 
spatial mean and oscillating part, by writing
\begin{equation}
\label{eq:phi0_phiperp} 
 \phi(t,x) = \phi_0(t) e_0(x) + \phi_\perp(t,x)\;,
 \qquad
 \int_{\T} \phi_\perp(t,x) \6x = 0\;.
\end{equation} 
One then finds that the SPDE~\eqref{eq:SPDE_bif} is equivalent to the coupled 
SDE--SPDE system  
\begin{align}
\label{eq:phi0_phiperp_stoch} 
\6\phi_0(t) 
&= \frac{1}{\eps} \biggbrak{g(t) - \phi_0(t)^2 - b(t,\phi_0(t)e_0) 
 + b_0(t,\phi_0(t),\phi_\perp(t,\cdot))}\6t + 
\frac{\sigma}{\sqrt{\eps}}\6W_0(t)\;,\\
 \6\phi_\perp(t,x) 
&= \frac{1}{\eps} \biggbrak{\Delta\phi_\perp(t,x) + a(t,\phi_0(t)) 
\phi_\perp(t,x) 
 + b_\perp(t,\phi_0(t),\phi_\perp(t,\cdot))}\6t + 
\frac{\sigma}{\sqrt{\eps}}\6W_\perp(t,x)\;,
 \end{align} 
where $W_0(t)$ is a standard Brownian motion, $W_\perp(t,x)$ is an independent 
zero-mean space-time white noise, 
\begin{equation}
\label{eq:def_a} 
 a(t,\phi_0) = -2 \phi_0 - \frac{1}{\sqrt{L}}\partial_\phi b(t,\phi_0 e_0)\;,
\end{equation} 
while $b_0$ and $b_\perp$ are (non-local) remainders specified 
in~\eqref{eq:b0_a_bperp} below. 

\begin{prop}[Deterministic dynamics near the origin]
\label{eq:prop_phi0_det} 
The deterministic equation given by \eqref{eq:SPDE_bif} with $\sigma = 0$ 
admits a particular solution satisfying $\phi_\perp(t,x) = 0$, while $\phi_0$ 
obeys the ordinary differential equation 
\begin{equation}
\label{eq:phi0_det} 
 \eps\dot\phi_0(t) 
 = g(t) - \phi_0(t)^2 - b(t,\phi_0(t)e_0)\;.
\end{equation} 
\end{prop}
The equation~\eqref{eq:phi0_det} for $\phi_0(t)$ is exactly of the form 
previously analysed in the work~\cite{BG_SR}. In particular, Theorem~2.5 in 
that article states that there exists a particular solution $\bar\phi_0(t)$ 
tracking $\phi^\star_+(t)$, in the sense that there are constants $T_0, c_0 > 0$ 
such that 
\begin{equation}
 \bar\phi_0(t) - \phi^*_+(t) \asymp 
 \begin{cases}
  \displaystyle\frac{\eps}{\abs{t}} 
  &\text{for $-T_0 \leqs t \leqs -c_0(\sqrt{\delta\vee\eps})$}\;, 
\\[14pt]
  \displaystyle-\frac{\eps}{\abs{t}} 
  &\text{for $c_0\sqrt{\delta\vee\eps} \leqs t \leqs T_0$}   
 \end{cases}
\end{equation}
(\figref{fig:det_bif}). 
Furthermore, one has 
\begin{equation}
 \bar\phi_0(t) \asymp \sqrt{\delta\vee\eps}
 \qquad 
 \text{for $\abs{t} \leqs c_0\sqrt{\delta\vee\eps}$\;.}
\end{equation} 
As a consequence, the linearisation
\begin{align}
 \bar a(t,\bar\phi_0(t)) 
 &= \partial_\phi \bigbrak{g(t) - \phi^2 - b(t,\phi)} 
 \Bigr|_{\phi = \bar\phi_0(t)} \\
 &= -2 \bar\phi_0(t) - \partial_\phi b(t,\bar\phi_0(t)) 
\end{align} 
satisfies 
\begin{equation}
 \bar a(t,\bar\phi_0(t))  \asymp - \bigpar{\abs{t} \vee \sqrt{\delta\vee\eps}\,}
\end{equation} 
for all $t\in [-T_0, T_0]$. By a symmetry argument, similar results, with some 
signs reversed, hold for a particular solutions $\hat\phi_0(t)$ tracking the 
unstable equilibrium branch $\phi^*_-(t)$. 

Let $\zeta(t)$ be the solution of 
\begin{equation}
 \eps \dot\zeta(t) = 2\bar a(t,\bar\phi_0(t)) \zeta(t) + 1
\end{equation} 
with initial condition $\zeta(-T_0) = (2\abs{\bar 
a(-T_0,\bar\phi_0(-T_0))})^{-1}$. This 
function is related to the variance of the linearisation around 
$\bar\phi_0(t)$ of the equation for $\phi_0$. It can be written explicitly as 
\begin{align}
 \zeta(t) &= \frac{1}{2\abs{\bar a(-T_0,\bar\phi_0(-T_0))}} 
\e^{2\bar\alpha(t,-T_0)/\eps}
 + \frac{1}{\eps} \int_{-T_0}^t \e^{2\bar\alpha(t,t_1)/\eps} \6t_1\;, \\ 
 \bar\alpha(t,t_1) &= \int_{t_1}^t \bar a(u,\bar\phi_0(u))\6u\;.
\end{align} 
However, it is more important for what follows to know that 
\begin{equation}
\label{eq:zeta} 
 \zeta(t) \asymp \frac{1}{\abs{\bar a(t,\bar\phi_0(t))}}
 \asymp \frac{1}{\abs{t}\vee \sqrt{\delta\vee\eps} }
 \qquad 
 \forall t\in[-T_0,T_0]\;,
\end{equation} 
see~\cite[Equation~(4.18)]{BG_SR}. 
With these notations in place, we are able to define the sets 
\begin{align}
 \cB_0(h) 
 &= \Bigset{(t,\phi_0) \colon t\in [-T_0,T_0], 
 \abs{\phi_0 - \bar\phi_0(t)} < h\sqrt{\zeta(t)}}\;, \\
 \cB_\perp(h_\perp) 
 &= \Bigset{(t,\phi) \colon t\in [-T_0,T_0], 
 \norm{\phi_\perp}_{H^s} < h_\perp}\;,
\end{align} 
where $s\in(0,\frac12)$, and $h, h_\perp > 0$. 
The exit from $\cB_\perp(h_\perp)$ is described by the following analogue of 
Theorem~\ref{thm:stable}.

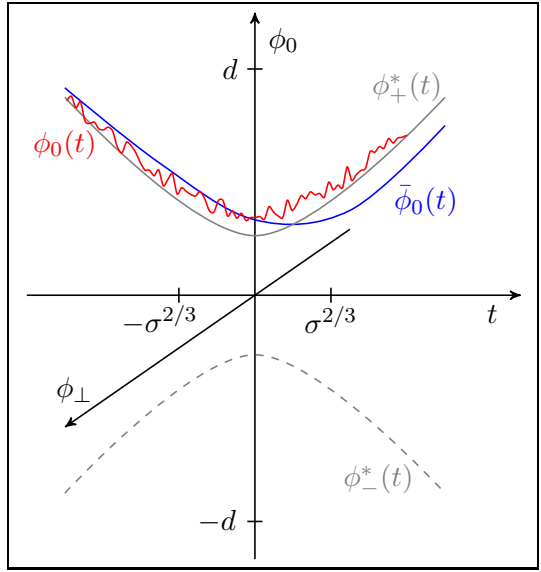
\begin{figure}
\begin{center}
\fbox{
\begin{tikzpicture}
[>=stealth',point/.style={circle,inner sep=0.035cm,fill=white,draw},
encircle/.style={circle,inner sep=0.07cm,draw},
x=2.5cm,y=2.5cm,declare function={f(\x) = sqrt((\x)^2+0.25)-0.065;
g(\x) = sqrt(\x^2 + 0.04);}]

\draw[->,semithick] (-1.2,0) -> (1.4,0);
\draw[->,semithick] (0,-1.4) -> (0,1.5);
\draw[<-,semithick] (-1,-0.7) -> (0.5,0.35);


\draw[semithick,blue] plot[smooth,tension=.6]
  coordinates{(-1,1.1) (-0.5,0.7) (0,0.4) (0.5,0.45) (1,0.9)};

\draw[semithick,gray] plot[smooth,tension=.6]
  coordinates{(-1,1.05) (0,0.316) (1,1.05)};

\draw[semithick,gray,dashed] plot[smooth,tension=.6]
  coordinates{(-1,-1.05) (0,-0.316) (1,-1.05)};

\draw[red,semithick,-,smooth,domain=-0.99:0.8,samples=75,/pgf/fpu,
/pgf/fpu/output format=fixed] plot ({\x}, {f(\x) +0.05*rand});

\node[] at (1.25,-0.1) {$t$};
\node[] at (0.15,1.35) {$\phi_0$};
\node[] at (-0.95,-0.5) {$\phi_\perp$};

\node[blue] at (0.9,0.5) {$\bar \phi_0(t)$};
\node[gray] at (0.65,-0.98) {$\phi^*_-(t)$};

\node[gray] at (0.8,1.1) {$\phi^*_+(t)$};
\node[red] at (-1.0,0.8) {$\phi_0(t)$};

\draw[semithick] (0.4,0.04) -- (0.4,-0.04);
\draw[semithick] (-0.4,0.04) -- (-0.4,-0.04);
\draw[semithick] (-0.04,1.2) -- (0.04,1.2);
\draw[semithick] (-0.04,-1.2) -- (0.04,-1.2);

\node[] at (0.4,-0.13) {$\sigma^{2/3}$};
\node[] at (-0.5,-0.13) {$-\sigma^{2/3}$};
\node[] at (-0.13,1.2) {$d$};
\node[] at (-0.2,-1.2) {$-d$};

\end{tikzpicture}
}
\vspace{-3mm}
\end{center}
\caption{Weak noise regime $\sigma\ll(\delta\vee\eps)^{3/4}$. The equilibrium 
branches $\phi^*_\pm(t)$, as well as the deterministic solution $\bar\phi_0(t)$, 
belong to the hyperplane $\set{\phi_\perp = 0}$, while $\phi_0(t)$ denotes the 
projection of the solution $\phi(t,x)$ on this hyperplane. }
\label{fig:weak_noise} 
\end{figure}

\begin{theorem}[Transverse stochastic dynamics for $\phi_\perp$]
\label{thm:phiperp} 
If $T_0$ is sufficiently small, then for any $s\in(0,\frac12)$ and any $\nu > 
0$, there exist constants $\kappa = \kappa(s), \eps_0, h^0_\perp$ and  
$C(\kappa,t,\eps,s) > 0$ such that, whenever $0 < \eps \leqs \eps_0$ and $0 < 
h_\perp \leqs h^0_\perp \eps^{\nu}$, the solution of~\eqref{eq:SPDE_bif} with 
initial condition $\phi(-T_0,\cdot) = \bar\phi_0(-T_0)e_0$ satisfies 
\begin{equation}
\bigprob{\tau_{\cB_\perp(h_\perp)} < t \wedge \tau_{\cB_0(h)}}
\leqs C(\kappa,t,\eps,s) \exp\biggset{-\kappa\frac{h_\perp^2}{2\sigma^2} 
\biggbrak{1-\biggOrder{\frac{h_\perp}{\eps^\nu}}}}\;.
\end{equation} 
The result remains true when $\tau_{\cB_0(h)}$ is replaced by 
$\inf\set{t\in[-T_0,T_0] \colon \abs{\phi_0(t)} > d}$ for any sufficiently 
small $d$ of order $1$.
\end{theorem}

As before, the result also holds for initial conditions with a transverse part 
$\phi_\perp(-T_0,\cdot)$ having $H^s$ norm up to order $1$, provided one 
considers the probability of leaving $\cB_\perp(h_\perp)$ after a time of order 
$\eps\log(\norm{\phi_\perp(-T_0,\cdot)}_{H^1}h_\perp^{-1}))$. 

On the other hand, the exit from $\cB_0(h)$ is described by the 
following result. 

\begin{theorem}[Stochastic dynamics near $\bar\phi_0(t)$]
\label{thm:phi0_stable} 
For any $t\in[-T_0, T_0]$, let 
\begin{equation}
 \hat\zeta(t) = \sup_{-T_0 \leqs s \leqs t} \zeta(s)\;.
\end{equation} 
Then there exist constants $\eps_0, h_0, c_\perp, \kappa > 0$ such 
that, whenever $0 < \eps \leqs \eps_0$, $0 < h \leqs h_0 
\hat\zeta(t)^{-3/2}$ and $0 < h_\perp <  c_\perp h \hat\zeta(t)^{1/2}$, the 
solution of~\eqref{eq:SPDE_bif} with initial condition 
$\phi(-T_0,\cdot) = \bar\phi_0(-T_0)e_0$ satisfies 
\begin{equation}
\bigprob{\tau_{\cB_0(h)} < t \wedge \tau_{\cB_\perp(h_\perp)}}
\leqs C(t,\eps) \exp\biggset{-\kappa\frac{h^2}{2\sigma^2}}\;,
\end{equation} 
where $\kappa = 1 - \Order{h\hat\zeta(t)^{3/2}}$
and $C(t,\eps) = \frac{\bar\alpha(t,-T_0)}{\eps^2} + 2$.
As before, the bound extends to general initial conditions in 
$\cB_\perp(h_\perp)$ with $(-T_0,\phi_0(-T_0))$ in $\cB_0(h)$.  
\end{theorem}

One consequence of this result is that there are two qualitatively different 
regimes, depending on the noise intensity:
\begin{itemize}
\item 	\textit{Weak-noise regime:} if $\sigma\ll(\delta\vee\eps)^{3/4}$, 
Theorem~\ref{thm:phi0_stable} can be applied for any $t\in[-T_0,T_0]$, and 
shows that $\phi_0(t)$ remains close to $\bar\phi_0(t)$ with high probability 
during the whole time interval (\figref{fig:weak_noise}).

\item 	\textit{Strong-noise regime:} if $\sigma\gg(\delta\vee\eps)^{3/4}$, 
Theorem~\ref{thm:phi0_stable} can only be applied up to times $t$ of order 
$-\sigma^{2/3}$, showing that $\phi_0(t)$ is unlikely to become negative up to 
times of that order (\figref{fig:strong_noise}). 
\end{itemize}

\begin{figure}
\begin{center}
\fbox{
\begin{tikzpicture}
[>=stealth',point/.style={circle,inner sep=0.035cm,fill=white,draw},
encircle/.style={circle,inner sep=0.07cm,draw},
x=2.5cm,y=2.5cm,declare function={f(\x) = sqrt((\x)^2+0.15)-0.1;
g(\x) = 0.1 -\x - sqrt(\x^2 + 0.1);}]

\draw[->,semithick] (-1.2,0) -> (1.4,0);
\draw[->,semithick] (0,-1.4) -> (0,1.5);
\draw[<-,semithick] (-1,-0.7) -> (0.5,0.35);

\draw[semithick,blue] plot[smooth,tension=.6]
  coordinates{(-1,1.1) (-0.5,0.7) (0,0.4) (0.5,0.45) (1,0.9)};

\draw[semithick,gray] plot[smooth,tension=.6]
  coordinates{(-1,1.05) (0,0.316) (1,1.05)};

\draw[semithick,gray,dashed] plot[smooth,tension=.6]
  coordinates{(-1,-1.05) (0,-0.316) (1,-1.05)};

\draw[red,semithick,-,smooth,domain=-1.3:0.9,samples=75,/pgf/fpu,
/pgf/fpu/output format=fixed] plot ({0.5*g(\x)}, {1.1*\x 
+0.1*rand});

\node[] at (1.25,-0.1) {$t$};
\node[] at (0.15,1.35) {$\phi_0$};
\node[] at (-0.95,-0.5) {$\phi_\perp$};

\node[blue] at (0.9,0.5) {$\bar \phi_0(t)$};
\node[gray] at (0.65,-0.98) {$\phi^*_-(t)$};

\node[gray] at (0.8,1.1) {$\phi^*_+(t)$};
\node[red] at (-1.0,0.8) {$\phi_0(t)$};

\draw[semithick] (0.4,0.04) -- (0.4,-0.04);
\draw[semithick] (-0.4,0.04) -- (-0.4,-0.04);
\draw[semithick] (-0.04,1.2) -- (0.04,1.2);
\draw[semithick] (-0.04,-1.2) -- (0.04,-1.2);

\node[] at (0.4,-0.13) {$\sigma^{2/3}$};
\node[] at (-0.5,-0.13) {$-\sigma^{2/3}$};
\node[] at (-0.13,1.2) {$d$};
\node[] at (-0.2,-1.2) {$-d$};
\end{tikzpicture}
}
\end{center}
\vspace{-3mm}
\caption{Strong noise regime $\sigma\gg(\delta\vee\eps)^{3/4}$. Solutions are 
likely to cross the unstable equilibrium branch $\phi^*_-(t)$.}
\label{fig:strong_noise} 
\end{figure}
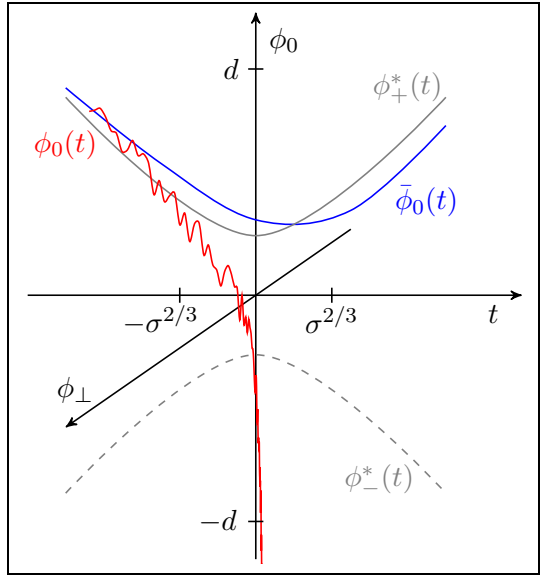

The behaviour in the strong-noise regime for times $t \geqs -c_1\sigma^{2/3}$ 
is described by the following theorem.

\begin{theorem}[Strong-noise regime]
\label{thm:phi0_strong} 
Fix sufficiently small constants $d, c_1>0$. Let $h>0$ be such that
\begin{equation}
\label{eq:phibar_upper_limit} 
 \bar\phi_0(t) + h\sqrt{\zeta(t)} \leqs d
 \qquad \forall t\in[-c_1\sigma^{2/3},c_1\sigma^{2/3}]\;.
\end{equation} 
Then there exist constants $\kappa, \bar c_\perp > 0$ 
such that for 
\begin{equation}
 0 < h_\perp < \bar c_\perp \Bigbrak{\sigma^{2/3} \wedge 
\sqrt{h}\,\hat\zeta(t)^{-1/4}}\;,
\end{equation} 
any solution of~\eqref{eq:SPDE_bif} starting at time $-c_1\sigma^{2/3}$ with an 
initial value $\phi_0$ belonging to the interval 
$(-d,\bar\phi_0(-c_1\sigma^{2/3} + \frac12 h 
\sqrt{\zeta(-c_1\sigma^{2/3})}\,)]$ satisfies 
\begin{align}
\Bigprob{\phi_0(t_1) > -d \; \forall t_1\in[-c_1\sigma^{2/3},t \wedge  
\tau_{\cB_\perp(h_\perp)}}
\leqs{}& 
\frac32 
\exp\biggset{-\kappa\frac{\hat\alpha(t,-c_1\sigma^{2/3})}
{\eps\log(\sigma^{-1})}}
\\ 
&{}+ C(t,\eps) \e^{-\kappa h^2/\sigma^2}
\label{eq:prob_phi0_strong} 
\end{align}
for all $t\in[-c_1\sigma^{2/3},c_1\sigma^{2/3}]$, where 
\begin{equation}
 \hat\alpha(t,t_1) = \int_{t_1}^t \hat a(t_2,\hat\phi_0(t_2))\6t_2\;, 
 \qquad 
 \hat a(t,\hat\phi_0(t)) 
 = \partial_\phi \bigbrak{g(t) - \phi^2 + b(t,\phi)} 
 \Bigr|_{\phi = \hat\phi_0(t)}
\end{equation} 
and $C(t,\eps)= \frac{\abs{\bar\alpha(t,-c_1\sigma^{2/3})}}{\eps^2} + 2$.
\end{theorem}

\begin{remark}
The condition~\eqref{eq:phibar_upper_limit} is required since we did not make 
any assumptions on the behaviour of $f$ for $x\geqs d$. For instance, our 
results apply if there exist more equilibrium branches above $d$. If, however, 
there are no such branches, as in the case of the Allen--Cahn equation with 
drift term~\eqref{eq:f_AC}, this condition can probably be relaxed. 
\end{remark}

To complete the description of the dynamics, we also need to show that once the 
process has reached level $-d$, it is also likely to reach a neighbourhood of 
the next stable equilibrium branch, where one can then apply 
Theorem~\ref{thm:stable} to describe the dynamics up to the next (avoided) 
bifurcation point (\figref{fig:strong_d0}). This can be easily done via the 
following analogue of~\cite[Proposition~4.7]{BG_SR}. 

\begin{prop}[Reaching level $-d_0 < -d$]
\label{prop:d0} 
There exists a constant $M>0$ such that if the drift 
term $f$ satisfies 
\begin{equation}
 f(t,\phi) \leqs -f_0 - Mh_\perp^2
 \qquad \forall (t,x) \in [-T_0,T_0] \times [-d_0, -d + \rho]
\end{equation} 
for some constants $d_0, f_0 > 0$ and $\rho \in (0,d)$, then there exist 
constant $\tilde c, \tilde \kappa > 0$ such that for all $t_0\in[-T_0,T_0 - 
\tilde c\eps]$, the solution of~\eqref{eq:SPDE_bif} with initial condition 
$\phi_0(t_0) = -d$ satisfies 
\begin{equation}
 \bigprob{\phi_0(t_1) > -d_0 \; \forall t_1 \in [t_0,(t_0 + 
\tilde c\eps)\wedge\tau_{\cB_\perp}(h_\perp)]} 
\leqs \e^{-\tilde\kappa/\sigma^2}\;.
\end{equation}
\end{prop}

\begin{figure}
\begin{center}
\fbox{
\begin{tikzpicture}
[>=stealth',point/.style={circle,inner sep=0.035cm,fill=white,draw},
encircle/.style={circle,inner sep=0.07cm,draw},
x=2.5cm,y=2.5cm,declare function={f(\x) = sqrt((\x)^2+0.15)-0.1;
g(\x) = 0.1 -\x - sqrt(\x^2 + 0.1);}]

\draw[->,semithick] (-1.2,0) -> (1.4,0);
\draw[->,semithick] (0,-2) -> (0,1.5);
\draw[<-,semithick] (-1,-0.7) -> (0.5,0.35);

\draw[semithick,blue] plot[smooth,tension=.6]
  coordinates{(-1,1.1) (-0.5,0.7) (0,0.4) (0.5,0.45) (1,0.9)};

\draw[semithick,gray] plot[smooth,tension=.6]
  coordinates{(-1,1.05) (0,0.316) (1,1.05)};

\draw[semithick,gray] plot[smooth,tension=.6]
  coordinates{(-1,-1.55) (0,-1.85) (1,-1.55)};

\draw[semithick,gray,dashed] plot[smooth,tension=.6]
  coordinates{(-1,-1.05) (0,-0.316) (1,-1.05)};

\draw[red,semithick,-,smooth,domain=0.04:0.9,samples=75,/pgf/fpu,
/pgf/fpu/output format=fixed] (0.04, -1.9) -- plot ({\x}, {-1.85 + 0.35*\x*\x 
+0.1*rand});
\draw[red,semithick,-,smooth,domain=-1.7:1.02,samples=75,/pgf/fpu,
/pgf/fpu/output format=fixed] plot ({0.5*g(\x)}, {1.1*\x 
+0.2*rand});

\node[] at (1.25,-0.1) {$t$};
\node[] at (0.15,1.35) {$\phi_0$};
\node[] at (-0.95,-0.5) {$\phi_\perp$};

\node[blue] at (0.9,0.5) {$\bar \phi_0(t)$};
\node[gray] at (0.65,-0.98) {$\phi^*_-(t)$};
\node[gray] at (0.85,-1.8) {$\phi^*(t)$};
\node[gray] at (0.8,1.1) {$\phi^*_+(t)$};
\node[red] at (-1.0,0.8) {$\phi_0(t)$};

\draw[semithick] (0.4,0.04) -- (0.4,-0.04);
\draw[semithick] (-0.4,0.04) -- (-0.4,-0.04);
\draw[semithick] (-0.04,1.2) -- (0.04,1.2);
\draw[semithick] (-0.04,-1.2) -- (0.04,-1.2);
\draw[semithick] (-0.04,-1.6) -- (0.04,-1.6);

\node[] at (0.4,-0.13) {$\sigma^{2/3}$};
\node[] at (-0.5,-0.13) {$-\sigma^{2/3}$};
\node[] at (-0.13,1.2) {$d$};
\node[] at (-0.2,-1.2) {$-d$};
\node[] at (-0.2,-1.6) {$-d_0$};
\end{tikzpicture}
}
\end{center}
\vspace{-3 mm}\caption{Strong noise regime, behaviour after reaching level 
$-d$. If the drift term is negative, bounded away from zero, in an interval 
$[-d_0,-d]$, solutions are likely to reach $-d_0$ after another time of order 
$\eps$.}
\label{fig:strong_d0} 
\end{figure}
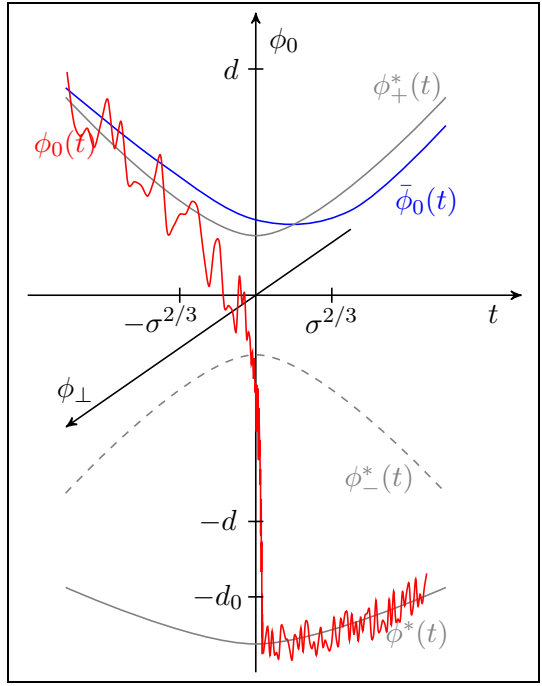


\subsection{Discussion}
\label{ssec:results_discussion} 

Let us first consider the weak-noise regime $\sigma \ll 
(\delta\vee\eps)^{3/4}$. For Theorems~\ref{thm:phiperp} 
and~\ref{thm:phi0_stable} to yield useful results, we need 
\begin{equation}
\label{eq:hhperp_min} 
 h \gg \sigma\;, \qquad h_\perp \gg \sigma\;.
\end{equation} 
If $t \geqs 0$, then $\hat\zeta(t) \asymp (\delta\vee\eps)^{-1/2}$. For the 
theorems to be applicable, we then need the conditions 
\begin{equation}
 h \lesssim (\delta\vee\eps)^{3/4}\;, \qquad 
 h_\perp \lesssim h (\delta\vee\eps)^{-1/4} \wedge \eps^\nu\;,
\end{equation} 
where $\nu>0$ can be chosen arbitrarily small. The weak-noise condition implies 
that all conditions on $h$ and $h_\perp$ can indeed be met simultaneously. In 
particular, since the minimal value of $\bar\phi_0(t)$ has order 
$(\delta\vee\eps)^{1/2}$, we can take $h$ of order $(\delta\vee\eps)^{3/4}$, 
and $h_\perp$ of order $(\delta\vee\eps)^{1/2} \wedge \eps^\nu$. We thus obtain 
\begin{align}
 \bigprob{\exists t\in[-T_0,T_0] \colon \phi_0(t) < 0}
 &\leqs \bigprob{\tau_{\cB_0(h)} \wedge \tau_{\cB_\perp(h_\perp)} < T_0} \\
 &= \bigprob{\tau_{\cB_\perp(h_\perp)} < T_0\wedge \tau_{\cB_0(h)}}
 + \bigprob{\tau_{\cB_0(h)} < T_0\wedge \tau_{\cB_\perp(h_\perp)}} \\
 &\leqs C_1(\eps) \exp\biggset{-\kappa\frac{(\delta\vee\eps)^{3/2} 
\wedge \eps^{2\nu}}{\sigma^2}}\;.
\label{eq:summary_weak} 
\end{align}
The term $\eps^{2\nu}$ can be disregarded as soon as $\delta$ is sufficiently 
small. In other words, the probability of making a transition from a 
neighbourhood of the stable branch $\phi^*_+$ to the unstable branch $\phi^*_-$ 
or to the other stable branch is exponentially small, with a parameter of order 
$(\delta\vee\eps)^{3/2}/\sigma^2$. 

Consider now the strong-noise regime $\sigma\geqs(\delta\vee\eps)^{3/4}$. We 
still require the conditions~\eqref{eq:hhperp_min} to hold, but modify the 
upper bounds on $h$ and $h_\perp$. As long as $t < -c_1\sigma^{2/3}$, 
Theorem~\ref{thm:phi0_stable} can be applied with $h \lesssim \abs{t}^{3/2}$, 
yielding 
\begin{equation}
 \bigprob{\tau_{\cB_0(h)} \wedge \tau_{\cB_\perp(h_\perp)} < t} 
 \leqs C(t,\eps) \exp\biggset{-\kappa'\frac{\abs{t}^3}{\sigma^2}}
\end{equation} 
for some $\kappa'>0$. This shows in particular that $\phi_0$ is unlikely to 
reach $0$ before times of order $-\sigma^{2/3}$. 

To see what happens for larger times, we do no longer use 
Theorem~\ref{thm:phi0_stable}, but only Theorems~\ref{thm:phiperp} 
and~\ref{thm:phi0_strong}, applied to an interval of the form 
$[-c_1\sigma^{2/3},-c_2\sigma^{2/3}]$. Then 
$\hat\alpha(-c_2\sigma^{2/3},-c_1\sigma^{2/3})$ has order $\sigma^{4/3}$ and
the conditions on $h$ and $h_\perp$ can be summarised as
\begin{equation}
 h \lesssim \sigma^{1/3}\;, \qquad 
 h_\perp \lesssim \sqrt{h} \sigma^{-1/3} 
 \wedge \sigma^{2/3} \wedge \eps^\nu\;. 
\end{equation} 
In particular, it is possible to take $h$ of order $\sigma^{1/3}$ and $h_\perp$ 
of order $\sigma^{2/3} \wedge \eps^\nu$. This yields 
\begin{equation}
\label{eq:summary_strong} 
 \bigprob{ \phi_0(t) > -d_0 \; \forall t\in[-c_1\sigma^{2/3},-c_2\sigma^{2/3}]} 
 \leqs \frac32 \exp\biggset{-\kappa'\frac{\sigma^{4/3}}{\eps\log(\sigma^{-1})}}
 + \biggOrder{\frac{\sigma^{4/3}}{\eps^2} \e^{-\kappa'/\sigma^{4/3}}}\;.
\end{equation} 
To summarise, we have thus obtained that with a probability exponentially close 
to $1$, the transverse component $\phi_\perp$ of the solution remains small in 
$H^1$ norm, while the spatial mean $\phi_0$ behaves in the same way as the 
solution of the one-dimensional SDE studied in~\cite{BG_SR}. In particular, 
there exist a weak-noise regime in which transitions between stable equilibria 
are very unlikely, cf.~\eqref{eq:summary_weak}, and a strong-noise regime, in 
which transitions are very likely, see~\eqref{eq:summary_strong}. 

An interesting question that remains open so far, is what can be said on 
regimes where the periodic forcing has a smaller amplitude, so that one stays 
in the weak-noise regime, but transitions still become likely over very long 
time spans. In the one-dimensional case, very precise results on the 
distribution of transition times have been obtained, for instance, 
in~\cite{BG_periodic2,NB_slowly_oscillating_20}. Generalising these results to 
the infinite-dimensional situation would require a good understanding of the 
effect of the dynamics of $\phi_\perp$ on transition times. 


\section{Proofs: the stable case}
\label{sec:proof_stable} 


\subsection{Deterministic case}
\label{ssec:stable_deterministic}

In this subsection, we give the proof of Proposition~\ref{det_prop} on the 
deterministic dynamics near a stable equilibrium branch $\phi^*(t) e_0$. We 
thus consider the deterministic equation
\begin{equation}
\label{eq:PDE_det_stable} 
\eps \partial_t \phi(t,x) 
= \Delta\phi(t,x) + f(t,\phi(t,x))\;,
\end{equation} 
where $t\in I = [0,T]$ and $f$ satisfies Assumptions~\ref{assum:U_large_phi} 
and~\ref{assum:a(t)}. We are interested in the deviation from the 
equilibrium branch, given by the difference 
$\psi(t,\cdot)=\phi(t,\cdot)-\phi^*(t)e_0$. Using Taylor's formula to expand 
$f(t,\phi^*(t)e_0 + \psi)$, we obtain that $\psi$ satisfies the equation 
\begin{equation}
\label{eq:SPDE_det_psi}
\eps \partial_t \psi(t,x) 
= \Delta\psi(t,x) + a(t)\psi(t,x) + 
b(t,\psi(t,x))-\eps\frac{\6}{\6t}\phi^*(t)e_0(x)\;,
\end{equation} 
where 
\begin{align}
a(t) &= \partial_\phi f(t,\phi^*(t)e_0)\;, \\
b(t,\psi) &= \frac{1}{2}\partial_\phi^2 f\big(t,\phi^*(t)+\theta 
\psi\big)\psi^2 \qquad\text{ for some } \theta \in \brak{0,1}\;.
\end{align}
This shows in particular that there exist constants $d, M > 0$ such that
\begin{equation}
\label{eq:bound_btpsi} 
\bigabs{b(t,\psi)}\leqs M \psi^2\;,\qquad
\bigabs{\partial_\psi b(t,\psi)} \leqs M \abs{\psi} 
\end{equation} 
for all $t \in I$ and all $\psi\in\R$ such that $\abs{\psi} < d$. 

\begin{proof}[Proof of Proposition~\ref{det_prop}]
Following the main idea of the proof in~\cite{Tihonov} in the 
finite-dimensional case, we define a Lyapunov function 
\begin{equation}
 \label{eq:Lyapunov}
 V(\psi) = \frac12 \norm{\psi}_{H^1}^2 
 = \frac12 \norm{\psi}_{L^2}^2 + \frac{L^2}{2\pi^2} 
\norm{\nabla\psi}_{L^2}^2\;. 
\end{equation} 
Let $\pscal{\cdot}{\cdot}$ denote the $L^2$ inner product. Observing that 
$\norm{\nabla\psi}_{L^2}^2 = \pscal{\nabla\psi}{\nabla\psi} = 
-\pscal{\psi}{\Delta\psi}$, and using self-adjointness of the Laplacian, we 
obtain that the time derivative of the Lyapunov function along a solution 
of~\eqref{eq:PDE_det_stable} satisfies
\begin{align}
 \eps\frac{\6}{\6t} V(\psi(t,\cdot)) 
 ={}& \pscal{\psi}{\eps\partial_t\psi}
 - \frac{L^2}{\pi^2} \pscal{\Delta\psi}{\eps\partial_t\psi} \\
 ={}& \pscal{\psi}{\Delta\psi} + a(t) \norm{\psi}_{L^2}^2 + 
\pscal{\psi}{b(t,\psi)} - \eps \frac{\6}{\6t} \phi^*(t) \pscal{\psi}{e_0} \\
&{}- \frac{L^2}{\pi^2} \Bigbrak{\norm{\Delta\psi}_{L^2}^2 
+ a(t) \pscal{\Delta\psi}{\psi} 
+ \pscal{\Delta\psi}{b(t,\psi)}}\;.
\end{align}
In the last line, we have used the fact that $\pscal{\Delta\psi}{e_0} = 0$
(here and below, we sometimes write $\psi$ instead of $\psi(t,\cdot)$ in order 
not to overload the notation). Regrouping terms, and bounding some obviously 
negative terms above by zero, we get 
\begin{equation}
\label{eq:dtV_det_stable1} 
 \eps\frac{\6}{\6t} V (\psi) 
 \leqs 2a(t) V(\psi) + \pscal{\psi}{b(t,\psi)} 
- \frac{L^2}{\pi^2}\pscal{\Delta\psi}{b(t,\psi)} 
- \eps \frac{\6}{\6t} \phi^*(t) \pscal{\psi}{e_0}\;.
\end{equation} 
Let $C_0 > 0$ be a constant to be fixed below. Assume that 
$\norm{\psi(0,\cdot)}_{H^1} < C_0$, and define the first-exit time 
\begin{equation}
\bar\tau = \inf \bigset{t>0 \colon \norm{\psi(t,\cdot)}_{H^1} \geqs C_0}\;.
\end{equation} 
By convention, we set $\bar\tau = \infty$ whenever 
$\norm{\psi(t,\cdot)}_{H^1} < C_0$ for all $t\in I$. 
Thus, for all $t \leqs \bar\tau$ in $I$, we have $\norm{\psi(t,\cdot)}_{H^1} 
< C_0$.
By Sobolev's inequality, this implies that for these $t$, one has  
\begin{equation}
 \abs{\psi(t,x)} 
 \leqs \norm{\psi(t,\cdot)}_{L^\infty} 
 \leqs \CSob \norm{\psi(t,\cdot)}_{H^1} \leqs \CSob C_0
\end{equation} 
for all $x\in\R$ and some numerical constant $\CSob$. By~\eqref{eq:bound_btpsi}, 
provided $\CSob C_0 \leqs d$, it follows that 
\begin{equation}
 \abs{b(t,\psi(t,x))} 
 \leqs \CSob M \norm{\psi(t,\cdot)}_{H^1}^2\;,
\end{equation} 
and thus 
\begin{equation}
 \norm{b(t,\psi(t,\cdot))}_{L^2}^2 
 \leqs \CSob^2 M^2 L \norm{\psi(t,\cdot)}_{H^1}^4\;.
\end{equation} 
By the Cauchy--Schwarz inequality, we get 
\begin{equation}
 \bigabs{\pscal{\psi(t,\cdot)}{b(t,\psi(t,\cdot))}} 
 \leqs \norm{\psi(t,\cdot)}_{L^2} \norm{b(t,\psi(t,\cdot))}_{L^2}
 \leqs \CSob M L^{1/2} \norm{\psi(t,\cdot)}_{H^1}^3\;.
\end{equation} 
Furthermore, integration by parts and~\eqref{eq:bound_btpsi} yield 
\begin{align}
 \bigabs{\pscal{\Delta\psi(t,\cdot)}{b(t,\psi(t,\cdot))}} 
 &= \biggabs{\int_0^L \nabla\psi(t,x)^2 \partial_\psi b(t,\psi(t,x)) \6x} \\
 &\leqs M \norm{\psi(t,\cdot)}_{L^\infty} \norm{\nabla\psi(t,\cdot)}_{L^2}^2\\
 &\leqs \CSob M \norm{\psi(t,\cdot)}_{H^1}^3\;.
\end{align}
Finally, owing to the implicit function theorem and 
Assumption~\ref{assum:a(t)}, the derivative of $\phi^*(t)$ is bounded 
by a constant $c$, so that 
\begin{equation}
 \biggabs{\frac{\6}{\6t} \phi^*(t) \pscal{\psi(t,\cdot)}{e_0}}
 \leqs c \norm{\psi(t,\cdot)}_{L^2} 
 \leqs c \norm{\psi(t,\cdot)}_{H^1}\;. 
\end{equation} 
Plugging the last three estimates in~\eqref{eq:dtV_det_stable1}, since $a(t)$ 
is negative and bounded away from zero by Assumption~\ref{assum:a(t)}, we 
obtain that $V(t) = V(\psi(t,\cdot))$ satisfies 
\begin{align}
\eps \dot V 
&\leqs -C_1V + C_2 V^{3/2} + \eps C_3 V^{1/2} \\
&\leqs -C_1\biggbrak{1 - \frac{C_0^{1/2}C_2}{C_1}}V + \eps C_3 V^{1/2}
\label{borne_lyapunov_1}
\end{align}
for all $t \leqs \bar\tau$, and some constants $C_1, C_2, C_3 > 0$. Choosing 
$C_0$ such that $C_0^{1/2} \leqs \frac{C_1}{2C_2}$, we obtain 
\begin{equation}
 \eps \dot V \leqs -\frac12 C_1V + \eps C_3 V^{1/2}
\end{equation} 
for all $t \leqs \bar\tau$. Setting $V(t) = Z(t)^2$ and dividing by $2Z(t)$, we 
get
\begin{equation}
 \eps \dot Z \leqs -\frac14 C_1 Z + \frac12 C_3 \eps\;.
\end{equation} 
Since the variable $W = Z - \frac{2C_3}{C_1}\eps$ satisfies $\eps\dot W \leqs 
-\frac14C_1 W$,  Gronwall's inequality yields 
\begin{equation}
 W(t) \leqs W(0) \e^{-C_1t/(4\eps)}
\end{equation} 
for all $t \leqs \bar\tau$. Thus for any $W(0)$ of order $\eps$, we find that 
$Z(t)$ remains of order $\eps$ for all $t < \bar\tau$, and thus $V(t)$ remains 
of order $\eps^2$. Choosing $\eps_0$ small enough and $0 < \eps < \eps_0$, we 
obtain in particular that $V(t) < C_0$ for all $t < \bar\tau$, so that 
assuming $\bar\tau < T$ would lead to a contradiction. We conclude that 
$\bar\tau \geqs T$, showing that $V(t) = \Order{\eps^2}$ for all $t\in I$, 
which is the claimed result.
\end{proof}

\begin{remark}
Another choice of Lyapunov function would have been  
\begin{equation}
\label{eq:func_lyapunov2}
V(t,\psi)=\norm{\psi}^2_{H^1}  + \int_{0}^{L} U_1(t,\psi(x)) \6x\;,
\end{equation} 
where $U_1(t,\psi) = U(t,\phi^*(t)e_0+\psi)$ is a shifted version of the 
potential $U$ introduced in~\eqref{eq:def_U}. This function is useful to 
control the behaviour of solutions of large $H^1$-norm. Indeed, 
one can show that there exist constants $M_1, M_2 > 0$ such that 
\begin{equation}
\label{borne_lyapunov0} 
 - M_1 
 \leqs V(t,\psi)\leqs M_2 \Bigpar{1 + \norm{\psi}^{2p_0}_{H^1}} 
 \qquad \forall t\in I
\end{equation} 
for all $t\in I$, and that $V(t,\psi(t,\cdot))$ is decreasing at least 
exponentially fast when it is large.
\end{remark}


\subsection{Stochastic case}
\label{ssec:stable_stochastic} 

We turn now to the analysis of the stochastic equation~\eqref{eq:SPDE} with 
$\sigma > 0$. The equation for the deviation $\psi(t,x) = \phi(t,x) - 
\bar\phi(t,x)$ of the solution from the deterministic solution tracking 
the stable equilibrium branch $\phi^*(t)e_0$ reads  
\begin{equation}
\label{eq:SPDE_nolin}
\6\psi(t,x) =\frac{1}{\eps}\bigbrak{\Delta\psi(t,x)+\bar 
a(t)\psi(t,x)+b(t,\psi(t,x))}\6t+ \frac{\sigma}{\sqrt{\eps}}\6W(t,x)\;,
\end{equation} 
where 
\begin{equation}
 \bar a(t) = \partial_\phi f(t,\bar \phi(t,x))
\end{equation} 
and $b(t,\psi)$ denotes again a nonlinear term, satisfying bounds analogous 
to~\eqref{eq:bound_btpsi}. We will start by analysing the linear case where $b$ 
vanishes in Subsection~\ref{sssec:stable_linear}, before turning to the general 
nonlinear case in Subsection~\ref{sssec:stable_nonlinear}. 

\subsubsection{Linear case}
\label{sssec:stable_linear} 

We consider first the linear version of~\eqref{eq:SPDE_nolin} given by 
\begin{equation}
\label{eq:SPDE_lin}
\6\psi(t,x) =\frac{1}{\eps}[\Delta\psi(t,x)+\bar a(t)\psi(t,x)]\6t+ 
\frac{\sigma}{\sqrt{\eps}}\6W(t,x)\;.
\end{equation}
Denote the eigenvalues of $-\Delta$ by
\begin{equation}
\label{eq:ev_Laplacian} 
 \mu_k = \frac{k^2\pi^2}{L^2}\;,
 \qquad 
 k \in \Z\;.
\end{equation}
Projecting~\eqref{eq:SPDE_lin} on the $k$th basis vector $e_k$, we obtain 
\begin{equation}
\label{eq:SPDE_lin_k}
\6\psi_k(t)=\frac{1}{\eps}\bar 
a_k(t)\psi_k(t)\6t+\frac{\sigma}{\sqrt{\eps}}\6W_k(t)\;,
\end{equation}
where $\bar a_k(t)=-\mu_k+\bar a(t)$ and the $\set{W_k(t)}_{t\geqs0}$ are 
independent Wiener processes (see for instance~\cite{Jetschke_86}). 
The solution of~\eqref{eq:SPDE_lin_k} is a Gaussian process and can be 
represented by the Ito integral (cf.~Duhamel's principle)
\begin{equation}
\label{eq:SPDE_lin_k_sol}
 \psi_k(t)=\frac{\sigma}{\sqrt{\eps}}\int_0^t 
 \e^{\bar\alpha_k(t,t_1)/\eps}
 \6W_k(t_1),
\end{equation}
where $\bar\alpha_k(t,t_1)=\int_{t_1}^t\bar a_k(t_2)\6t_2$. Thus, for each time 
$t$, $\psi_k(t)$ is characterised by its mean being zero and its variance given 
by 
\begin{equation}
\label{eq:SPDE_lin_k_sol_var}
\variance{\{\psi_k(t)\}}=\frac{\sigma^2}{\eps}\int_0^t 
\e^{2\bar\alpha_k(t,t_1)/\eps}
\6t_1\;.
\end{equation}
We may further assume that there are positive constants $\bar a_\pm$ and $c_0^{\pm}$ such that for all $t\in I$ 
\begin{align}
\label{a_k(t)} 
-\bar a_+&\leqs\bar a(t)\leqs-\bar a_-\;,\\
c_0^-\jbrack{k}^2\leqs \mu_k+\bar a_-&\leqs\bigabs{\bar a_k(t)}\leqs \mu_k+\bar a_+\leqs c_0^+\jbrack{k}^2\;,
\end{align} 
and due to the implicit function theorem we also have the existence of a 
constant $C$ such that  
\begin{equation}
\bigabs{\bar a(t)'}\leqs C \qquad \forall t\in I = [0,T]\;.
\end{equation} 

\begin{lemma}[Bound on the variance]
There exists a constant $C_0>0$ such that the variance satisfies the bound 
\begin{equation}
\label{eq:bound_Vk} 
 \variance{\{\psi_k(t)\}}\leqs C_0\frac{\sigma^2}{\jbrack{k}^2}
 \qquad \forall t\in I\;.
\end{equation} 
\end{lemma}

\begin{proof}
Using integration by parts, we obtain 
\begin{align}
 \frac{\variance{\{\psi_k(t)\}}}{\sigma^2}
 &= \int_0^t \frac{1}{2\bar a_k(s)} \frac{2\bar a_k(s)}{\eps}
 \e^{2\bar\alpha_k(t,s)/\eps}
\6s \\
&=\frac{1}{2\bar a_k(t)}-\frac{1}{2\bar a_k(0)}
\e^{2\bar\alpha_k(t)/\eps}
+\frac{1}{2}\int_0^t\frac{\bar a'_k(s)}{\bar a_k(s)^2}
\e^{2\bar\alpha_k(t,s)/\eps}
\6s\;,
 \end{align}
where we write $\bar\alpha_k(t,0)=\bar\alpha_k(t)$ for brevity.
The absolute value of integral can be bounded by 
\begin{equation}
\int_0^t\frac{\bigabs{\bar a_k'(s)}}{(-\mu_k-\bar a_-)^2}
\e^{-2(\mu_k+\bar a_-)(t-s)/\eps}
\6s
 \leqs\frac{C}{2(\mu_k+\bar a_-)^3}\eps\;.
\end{equation}
Therefore,
\begin{equation}
\frac{\variance{\{\psi_k(t)\}}}{\sigma^2}
\leqs\frac{1+\mathcal{O}(\eps)}{2(\mu_k+\bar a_-)}\leqs 
C_0\frac{1}{\jbrack{k}^2},
\end{equation}
as claimed. 
\end{proof}

Since each $\psi_k(t)$ is a one-dimensional process, we can easily adapt 
Theorem~2.4 in~\cite{BG_pitchfork} to obtain the following estimate.

\begin{lemma}
\label{lem:psi_k} 
Fix $\gamma > 0$. Then for any $k\in\Z$, we have the bound 
\begin{equation}
 \biggprob{\sup_{t\in I}\abs{\psi_k(t)}\geqs h}
\leqs C_k(T,\eps)\exp\biggset{-\kappa\jbrack{k}^2\frac{h^2}{\sigma^2}}\;,
\end{equation} 
where $C_k(T,\eps)=\frac{2c_0^-\jbrack{k}^2}{\gamma\eps}T$ and 
$\kappa=\frac{\e^{-2\gamma}}{2C_0}$. 
\end{lemma}
\begin{proof}
As in \cite[Theorem~2.4]{BG_pitchfork}, we introduce  
a partition $0=u_0<u_1<...<u_N=T$ of $\brak{0,T}$ by requiring 
$\bar\alpha_k(u_{l+1},u_l)=-\gamma\eps$ for $1\leqs l\leqs 
N=\lfloor{c_0^-\jbrack{k}^2}T/({\gamma\eps})\rfloor$. The proof then follows by 
approximating the process by a martingale on each interval $[u_l,u_{l+1}]$ and 
using a Bernstein-type inequality that follows directly from Doob's 
submartingale inequality. 
\end{proof}

\begin{proof}[Proof of Theorem \ref{thm:stable} in the linear case]
Fix constants $\eta$, $\rho > 0$ and $s\in (0, \frac{1}{2})$ such that 
$s=\frac{1}{2}-\rho$. For every decomposition $h^2=\sum_{k\in\Z}h_k^2$ one has
\begin{align}
\bigprob{\tau_{\cB(h)} < t}
&=\biggprob{\sup_{t\in I}\norm{\psi(t,\cdot)}^2_{H^s}\geqs h^2}\;\\
&=\biggprob{\sup_{t\in I}\sum_{k\in\Z}\jbrack{k}^{2s}\abs{\psi_k(t)}^2\geqs h^2}\;\\
&\leqs \sum_{k\in\Z}\biggprob{\sup_{t\in I}\abs{\psi_k(t)}^2\geqs h_k^2\jbrack{k}^{-2s}}\;\\
&\leqs 
\sum_{k\in\Z}C_k(T,\eps)\exp\biggset{-\kappa\frac{h_k^2}{\sigma^2}\jbrack{k}^{
2-2s } }\;.
\end{align} 
Choosing 
\begin{equation}
\label{eq:choice_hk} 
 h_k^2=C(\eta,s)h^2\jbrack{k}^{-2+2s+\eta}\;,
\end{equation} 
the condition $h^2=\sum_{k\in\Z}h_k^2$ yields 
\begin{equation}
 C(\eta,s)=\frac{1}{\displaystyle\sum_{k\in\Z}\jbrack{k}^{-2+2s+\eta}}\;.
\end{equation} 
Since the Riemann zeta function $\zeta(v) =\sum_{n\geqs1}n^{-v}$ converges for 
$v>1$, we get 
\begin{equation}
\sum_{k\in\Z}\frac{1}{\jbrack{k}^{2-2s-\eta}}\leqs 
1 + 2\sum_{k=1}^\infty \frac{1}{k^{2-2s-\eta}}
= 1 + 2\zeta(2-2s-\eta)<\infty \qquad 
\forall 
0<\eta<2\rho\;.
\label{eq:Riemann} 
\end{equation} 
With $h_k$ given by~\eqref{eq:choice_hk} and $\eta$ satisfying this condition, 
we get
\begin{align}
\biggprob{\sup_{t\in I}\norm{\psi(t,\cdot)}^2_{H^s}\geqs h^2}\;
&\leqs \sum_{k\in\Z}C_k(T,\eps)\exp\biggset{-\kappa 
C(\eta,s)\frac{h^2}{\sigma^2}\jbrack{k}^{\eta}}\;\\
&=\alpha_T\sum_{k\in\Z}\jbrack{k}^2\e^{-\beta\jbrack{k}^{\eta}}\;,
\end{align} 
where we write $\alpha_T=\frac{2c_0^-}{\gamma\eps}T$ and $\beta=\kappa 
C(\eta,s)\frac{h^2}{\sigma^2}$ for simplicity.
In order to bound the sum, we write 
\begin{equation}
 f(x) = (1+x^2)\e^{-\beta(1+x^2)^{\eta/2}}\;.
\end{equation} 
Note that we may assume that $f$ is decreasing by taking 
$h/\sigma$ larger than an $\eta$-dependent constant of order $1$ (which we 
may do, because otherwise the result is trivially true). 
Therefore, we obtain 
\begin{equation}
\sum_{k\in\Z}f(k) 
= f(0) + 2 \sum_{k=1}^\infty f(k) 
\leqs \e^{-\beta} + 2\int_0^\infty f(x)\6x\;.
\end{equation} 
In what follows, we show that the integral
\begin{equation}
I=\int_0^\infty f(x)\6x=\int_0^\infty (1+x^2)\e^{-\beta(1+x^2)^{\eta/2}} \6x\;
\end{equation} 
is finite, and, more precisely, has order $\beta^{-1/2}\e^{-\beta}$. We first 
make the change of variable $y=\beta(1+x^2)^{\eta/2}$, yielding 
\begin{equation}
I=\frac{1}{\eta\beta^{4/\eta}}\int_\beta^\infty\e^{-y}\frac{y^{4/\eta-1}}{\sqrt{
\bigpar{\frac{y}{\beta}}^{2/\eta}-1}} \6y\;.
\end{equation} 
The further change of variable $y=\beta+z$ gives 
\begin{equation}
I=\frac{\e^{-\beta}}{\eta\beta^{4/\eta}}\int_0^\infty\e^{-z}\frac{(\beta+z)^{
4/\eta-1}}{\sqrt{(1+\frac{z}{\beta})^{2/\eta}-1}} \6z\;.
\end{equation} 
Using Taylor's formula, we get the lower bound 
\begin{equation}
\Bigpar{1+\frac{z}{\beta}}^{2/\eta}-1
\geqs \frac{2}{\eta}\frac{z}{\beta}\;.
\end{equation} 
Therefore, 
\begin{align}
I&\leqs 
\frac{\e^{-\beta}}{\eta\beta^{4/\eta}}\sqrt{\frac{\beta\eta}{2}}
\int_0^\infty\frac{\e^{-z}}{\sqrt{z}}(\beta+z)^{4/\eta-1} \6z\;\\
&= \frac{\e^{-\beta}}{\eta\beta^{4/\eta}}\sqrt{\frac{\beta\eta}{2}} 
\Biggbrak{\int_0^\beta\frac{\e^{-z}}{\sqrt{z}}(\beta+z)^{4/\eta-1} \6z + 
\int_\beta^\infty\frac{\e^{-z}}{\sqrt{z}}(\beta+z)^{4/\eta-1} \6z}\;\\
&\leqs 
\frac{\e^{-\beta}}{\sqrt{2\eta}\beta^{4/\eta-1/2}}(2\beta)^{4/\eta-1}
\int_0^\beta\frac{\e^{-z}}{\sqrt{z}}\6z + 
\frac{\e^{-\beta}}{\sqrt{2\eta}\beta^{4/\eta-1/2}}2^{4/\eta-1}
\int_\beta^\infty\frac{\e^{-z}}{\sqrt{z}}z^{4/\eta-1} \6z\;\\
&\leqs c_1(\eta)\frac{\e^{-\beta}}{\sqrt{\beta}} + 
c_2(\eta)\frac{\e^{-\beta}}{\beta^{4/\eta-1/2}}\;,
\end{align} 
where $c_1(\eta)$ and $c_2(\eta)$ are bounded uniformly in $\beta$, provided 
$\eta < 8$. It follows that
\begin{align}
\sum_{k\in\Z}\jbrack{k}^2\e^{-\beta\jbrack{k}^{\eta}}
&\leqs  \e^{-\beta} + 2c_1(\eta)\frac{\e^{-\beta}}{\sqrt{\beta}} + 
2c_2(\eta)\frac{\e^{-\beta}}{\beta^{4/\eta-1/2}}\;\\
&=\e^{-\kappa C(\eta,s)h^2/\sigma^2} \Bigbrak{1+\bar c_1(\eta)\frac{\sigma}{h} 
+ \bar c_2(\eta)\Bigpar{\frac{\sigma^2}{h^2}}^{4/\eta-1/2}}\;.
\end{align} 
We thus conclude that 
\begin{align}
\biggprob{\sup_{t\in I}\norm{\psi(t,\cdot)}^2_{H^s}\geqs h^2}
&\leqs \alpha_T\Bigbrak{1+\bar c_1(\eta)\frac{\sigma}{h} + \bar 
c_2(\eta)\Bigpar{\frac{\sigma^2}{h^2}}^{4/\eta-1/2}}\e^{-\kappa 
C(\eta,s)h^2/\sigma^2} \\
&=: C(\gamma,T,\eps,s) \e^{-\kappa 
C(\eta,s)h^2/\sigma^2}\;,
\label{th2.4_lin}
\end{align} 
where we can fix, for instance, $\eta = \rho = \frac12 - s$, which yields 
$C(\eta,s) = [1 + 2\zeta(\frac32-s)]^{-1}$ by~\eqref{eq:Riemann}.
\end{proof}


\subsubsection{Nonlinear case}
\label{sssec:stable_nonlinear}

We return now to the study of the general nonlinear 
equation~\eqref{eq:SPDE_nolin}. By Duhamel's principle, its solution satisfies 
the equation 
\begin{align}
 \psi(t,\cdot) 
 &= \frac{\sigma}{\sqrt{\eps}}\int_0^t 
\e^{\bar\alpha(t,t_1)/\eps}\e^{[(t-t_1)/\eps]\Delta} \6W(t_1,\cdot)
 + \frac{1}{\eps}\int_0^t \e^{\bar\alpha(t,t_1)/\eps}\e^{[(t-t_1)/\eps]\Delta} 
b(t_1,\psi(t_1,\cdot)) \6t_1\;, \\
 &= \psi^0(t,\cdot) + \psi^1(t,\cdot)\;.
 \label{eq:psi0psi1} 
\end{align} 
Here $\bar\alpha(t,t_1) = \int_{t_1}^t \bar a(u)\6u$, and $\e^{t\Delta}$ 
denotes the heat kernel.
We notice that $\psi^0(t,x)$ is the solution of the linear 
equation~\eqref{eq:SPDE_lin}, and therefore satisfies the estimate 
\eqref{th2.4_lin}. 

In what follows, we give some technical results that will be needed several 
times in order to show that $\psi^1(t,\cdot)$ belongs to a certain Sobolev 
space included in $H^s$.

\begin{lemma}
\label{lemma3.3}
Let the potential $U(t,\phi)$ satisfy Assumption~\ref{assum:U_large_phi}, and 
assume $\psi(t,\cdot)\in H^s$ for all $0<s<\frac{1}{2}$. Then 
\begin{equation}
 \beta(t) = b(t,\psi(t,\cdot))
\end{equation} 
belongs to $H^r$ for all $r < \frac12$. Furthermore, for all 
$r<\frac12 - (2p_0+1)(\frac12 -s)$,  
there exists $C(r,s)<\infty$ such that 
\begin{equation}
\label{eq:bound_norm_beta} 
\norm{\beta(t)}_{H^r}\leqs 
C(r,s)\max\{\norm{\psi}_{H^s}^2,\norm{\psi}^{2p_0+1}_{H^s}\}\;. 
\end{equation}
\end{lemma}

\begin{proof}
Consider first the case where $U$ is a polynomial in $\psi$ of degree $2p_0$. 
Then $f(t,\psi)$ and $\beta(t)$ are polynomials of degree $2p_0+1$. Applying 
Young's inequality~\eqref{ineq_young}, we obtain by induction that if 
$\psi(t,\cdot)\in H^{\frac12-\kappa}$ for a $\kappa>0$, then for any $k\geqs2$, 
$\psi(t,\cdot)^k\in H^{r}$ for any $r < \frac12-k\kappa$.  It follows that 
$\beta(t)\in H^r$ for all $r<\frac{1}{2}-(2p_0+1)\kappa$. Since $\kappa>0$ is 
arbitrary, we conclude that indeed $\beta(t)\in H^r$ for all $r < \frac12$. The 
bound~\eqref{eq:bound_norm_beta} is then a consequence 
of Young's inequality~\eqref{ineq_young}, the bound~\eqref{eq:bound_btpsi} on 
$b(t,\psi)$ for small $\psi$, and the fact that $\beta(t)$ is a polynomial of 
degree $2p_0+1$. 

Consider now the general case. By Assumption~\ref{assum:U_large_phi}, 
$f(t,\psi)$ and $\beta(t)$ are each the sum of a polynomial of degree $2p_0+1$ 
and a bounded function $g(t,\psi)$. For $0<s<2$ and $1\leqs p, q\leqs \infty$, 
consider the Besov space $\mathcal{B}^{s,q}_p$. Then, ~\cite{bourdaud} 
shows that there exists a constant $R(p,q,s,M)>0$ such that for 
all $\psi$ in the positive cone, $(\cB^{s,q}_p)^+$, we have
\begin{equation}
\label{ineq_bessov}
\norm{g\circ \psi}_{\mathcal{B}^{s,q}_p}\leqs R(p,q,s,M) 
\norm{\psi}_{\mathcal{B}^{s,q}_p}\;.
\end{equation} 
In particular, whenever $p=q=2$ the Besov space $\mathcal{B}^{s,2}_2$ is 
nothing but the Sobolev space $H^s$. Thus, if $\psi(t,\cdot)\in H^s$, then 
$g\circ\psi\in H^s$ and $\beta(t)\in H^r$ for all 
$r<\frac{1}{2}$.
\end{proof}

\begin{lemma}[Schauder-type estimate]
\label{lemma3.5}
Assume  $\beta\in H^r$ for some $r \in (0,\frac12)$. Then for all $q<r+2$, 
there exists a constant $M(q,r)<\infty$ such that 
\begin{equation}
\label{eq:Schauder1} 
\norm{\e^{t\Delta}\beta}_{H^q}\leqs 
M(q,r) t^{-\frac{q-r}{2}}\norm{\beta}_{H^r} 
\end{equation} 
for all $t>0$. 
\end{lemma} 
\begin{proof}
Let $\gamma=\frac{q-r}{2}$. Writing the Fourier expansion of $\beta$ as 
$\beta(x) = \sum_{k\in\Z} \beta_k e_k(x)$,
we have 
\begin{equation}
 \e^{t\Delta}\beta(x) = \sum_{k\in\Z} \e^{-\mu_k t}\beta_k e_k(x)\;,
\end{equation} 
where the $-\mu_k$ are the eigenvalues of the Laplacian, 
cf.~\eqref{eq:ev_Laplacian}. 
By definition of the fractional Sobolev norm, we obtain 
\begin{align}
\norm{t^\gamma e^{t\Delta}\beta}^2_{H^q}
&= \sum_k\jbrack{k}^{2q}t^{2\frac{q-r}{2}}\e^{-2\mu_kt}\beta_k^2\\
&\leqs \sum_k 
\big[\jbrack{k}^2t\big]^{q-r}\e^{-c_0^-\jbrack{k}^2t}\jbrack{k}^{2r}
\beta_k^2\;\\
&=\sum_k H\bigpar{\jbrack{k}^2t}\jbrack{k}^{2r}\beta_k^2\;,
\end{align}
where $H(z)=z^{q-r}\e^{-c_0^-z}$ reaches its maximum at 
$z^*=\frac{q-r}{c_0^-}$. Therefore, 
\begin{equation}
0\leqs H(z)
\leqs M(q,r)^2 
= H(z^*) 
= \biggpar{\frac{q-r}{c_0^-}}^{q-r}\e^{-(q-r)} 
\end{equation}
for all $z\geqs0$. We conclude that for all $t\in I$, 
\begin{equation}
\bignorm{t^\frac{q-r}{2} \e^{t \Delta}\beta}^2_{H^q}
 \leqs\sum_k M(q,r)^2\jbrack{k}^{2r}\beta^2_k
 =M(q,r)^2\norm{\beta}^2_{H^r}
\end{equation}
as claimed. 
\end{proof} 

Applying this result to the term $\psi^1(t,\cdot)$ defined 
in~\eqref{eq:psi0psi1}, we obtain the following key estimate. 

\begin{cor}
Assume there exists $r\in(0,\frac12)$ such that $\beta(t) \in H^r$ for all 
$t\in I$. Then for all $q<r+2$, there exists a constant $M'(q,r)<\infty$ such 
that for all $t\in I$, one has $\psi^1(t,\cdot)\in H^q$ and 
\begin{equation}
\norm{\psi^1(t,\cdot)}_{H^q}
\leqs M'(q,r)\eps^{\frac{q-r}{2}-1} 
\sup_{0\leqs t_1\leqs t}\norm{\beta(t_1)}_{H^r}\;.
\end{equation}
\end{cor}
\begin{proof}
Note that $\bar\alpha(t,t_1)\leqs -\frac{c_0^-}{2}(t-t_1)$ whenever $t_1\leqs 
t$. Furthermore, the previous result implies that for any $q < r+2$, one has 
\begin{equation}
\bignorm{\e^{(t/\eps)\Delta} 
\beta(t)}_{H^q}
\leqs M(q,r) \biggpar{\frac{\eps}{t}}^{\frac{q-r}{2}}
\norm{\beta(t)}_{H^r}\;.
\end{equation}
Therefore
\begin{align}
 \norm{\psi^1(t,x)}_{H^q}
   &\leqs \frac{1}{\eps}\int_0^t \e^{-c_0^-(t-t_1)/(2\eps)}
\bignorm{\e^{[(t-t_1)/\eps]\Delta}\beta(t_1)}_{H^q}\6t_1\;\\
&\leqs M(q,r) \eps^{\frac{q-r}{2}-1} 
\sup_{0\leqs t_1\leqs t}   \norm{\beta(t_1)}_{H^r}
\int_0^t (t-t_1)^{-\frac{q-r}{2}}\6t_1\;,
\end{align}
and the integral over $t_1$ is bounded whenever $q-r < 2$. 
\end{proof}

Now, if $s\leqs q$ then $ H^q\subset H^s$ and thus $\psi^1(t,\cdot)\in H^s$ 
whenever $\psi(t,\cdot)\in H^s$. With these results, we can now prove Theorem 
\ref{thm:stable} for the nonlinear case.

\begin{proof}[Proof of Theorem \ref{thm:stable}]
For every decomposition $h=h_0+h_1$ with $h_0, h_1 > 0$, one has 
\begin{align}
\bigprob{\tau_{\mathcal{B}(h)}<t}
={}&\biggprob{\sup_{0\leqs t\leqs 
T\wedge\tau_{\mathcal{B}(h)}}\norm{\psi(t,\cdot)}_{H^s}>h}\;\\ 
\leqs{}& \biggprob{\sup_{0\leqs t\leqs 
T\wedge\tau_{\mathcal{B}(h)}}\norm{\psi^1(t,\cdot)}_{H^s}
+ \norm{\psi^0(t,\cdot)}_{ H^s}>h }\;\\
\leqs{}& \biggprob{\sup_{0\leqs t\leqs T}\norm{\psi^0(t,\cdot)}_{H^s}>h_0}\\
&{} +\biggprob{\sup_{0\leqs t\leqs 
T\wedge\tau_{\mathcal{B}(h)}}\norm{\psi^1(t,\cdot)}_{H^s}>h_1, \sup_{0\leqs 
t\leqs T}\norm{\psi^0(t,\cdot)}_{H^s}\leqs h_0}\;.
\end{align}
The first term on the right-hand side can be estimated by~\eqref{th2.4_lin}. 
Furthermore, for all $t<\tau_{\mathcal{B}(h)}$, we have 
$\norm{\beta(t)}_{H^r}\leqs M\norm{\psi(t,\cdot)}_{H^s}^2\leqs Mh^2$, so that
\begin{equation}
  \norm{\psi^1(t,x)}_{H^q}
  \leqs  M'(q,r)\eps^{\frac{q-r}{2}-1}Mh^2\;.
\end{equation} 
Choosing $h_1=M'(q,r)\eps^{\frac{q-r}{2}-1}Mh^2$, we get 
\begin{equation}
\biggprob{\sup_{0\leqs t\leqs T\wedge\tau_{\mathcal{B}(h)}}
\norm{\psi^1(t,x)}_{H^s}>h_1, \sup_{0\leqs t\leqs 
T}\norm{\psi^0(t,x)}_{H^s}\leqs h_0}=0\;.
\end{equation}
We thus obtain the result by choosing
$h_0=h-h_1=h-M'(q,r)\eps^{\frac{q-r}{2}-1}Mh^2
=h(1-\mathcal{O}(h/\eps^{\nu}))$ and $\nu=1-\frac{q-r}{2}$.
 \end{proof}


\section{Proofs: bifurcations}
\label{sec:proof_bif} 

Before entering the detailed analysis, we make a preliminary change of 
variables yielding the form~\eqref{eq:phi0_phiperp_stoch} of the equations. 
Let $\alpha$, $\beta$, $\gamma \in\R$. Using the scaling $t=\alpha\bar t$, 
$x=\beta\bar x$ and $\phi=\gamma\bar\phi$ in~\eqref{eq:SPDE_bif}, we obtain the 
following SPDE. For all $\bar x\in\brak{0,\bar L=\frac{L}{\beta}}$, one has 
\begin{equation}
  \6 \bar \phi(\bar t,\bar x)
  =\frac{1}{\bar\eps}\Bigbrak{\Delta\bar \phi(\bar t,\bar x)+\bar\alpha g(\bar 
t)-\bar \beta \bar \phi(\bar t,\bar x)^2-\bar \gamma b(\bar t,\bar \phi(\bar 
t,\bar x))}\6 \bar t + \frac{\bar\sigma}{\sqrt{\bar\eps}} \6W(\bar t,\bar x)\;,
\end{equation}
where $\bar\eps=\frac{\beta^2}{\alpha}\eps$, 
$\bar\alpha=\frac{\alpha^2\beta^2}{\gamma}$, $\bar \beta=\gamma\beta^2$, $\bar 
\gamma=\gamma^2\beta^2$ and $\bar\sigma=\frac{\sqrt{\beta}}{\alpha\gamma}\sigma$ 
(below, we drop the bars in order not to overload the notation). We now apply 
the decomposition~\eqref{eq:phi0_phiperp} of the solution in its mean and 
oscillating part. Taylor's formula yields 
\begin{equation}
 b(t, \phi_0 e_0 + \phi_\perp) 
 = b(t,\phi_0e_0) + \partial_\phi b(t,\phi_0e_0) \phi_\perp
 + \frac12 \partial_\phi^2 b(t,\phi_0e_0) \phi_\perp^2 
 + R(t,\phi_0e_0,\phi_\perp)\;,
\end{equation} 
where
\begin{equation}
 R(t,\phi_0e_0,\phi_\perp)
 = \frac16 \partial_\phi^3 b(t,\phi_0e_0 + \theta\phi_\perp) \phi_\perp^3
\end{equation} 
for some $\theta\in\brak{0,1}$. 
Therefore, the spatially constant part $\phi_0(t)$ of the solution $\phi(t,x)$ 
satisfies the equation  
\begin{align}
\6\phi_0(t)
={}& \frac{1}{\eps}\pscal{e_0}{\Delta \phi(t,\cdot)+\alpha g(t)-\beta 
\phi(t,\cdot)^2-\gamma b(t,\phi(t,\cdot))} \6t 
+ \frac{\sigma}{\sqrt{\eps}} \pscal{e_0}{\6W( t, \cdot)}\;\\
={}& \frac{1}{\eps} \biggl[
\alpha\sqrt{L} \, g(t)-\frac{\beta}{\sqrt{L}}\phi_0(t)^2-\frac{\beta}{\sqrt{L} 
} \norm{\phi_\perp}_{L^2}^2 - \gamma\sqrt{L}\, b(t,\phi_0(t)e_0) \\
&{}\quad -\frac{\gamma}{2\sqrt{L}} 
\partial_\phi^2b(t,\phi_0(t)e_0)\norm{\phi_\perp(t,\cdot)}_{L^2}^2
-\gamma \bigpscal{e_0}{R(t,\phi_0(t)e_0,\phi_\perp(t,\cdot))}
 \biggr] \6t \\
&{}+ \frac{\sigma}{\sqrt{\eps}} \6W_0(t)\;.
\end{align}
On the other hand, the mean zero part $\phi_\perp(t,x) = \phi(t,x) - 
\phi_0(t)e_0(x)$ satisfies 
\begin{align} 
 \6 \phi_\perp(t,x)
={}& \6\phi(t,x)-\6 \phi_0(t)e_0(x)\;\\
={}& \frac{1}{\eps} \biggl[\Delta 
\phi_\perp(t,x)-\biggpar{2\frac{\beta}{\sqrt{L}}
\phi_0(t)+\gamma\partial_\phi 
b(t,\phi_0(t)e_0(x))}\phi_\perp(t,x) \\
&\quad{}- \biggpar{\beta+\frac{\gamma}{2}
\partial^2_\phi b(t,\phi_0(t)e_0(x))}
\biggpar{\phi_\perp(t,x)^2 - \frac1L \norm{\phi_\perp(t,\cdot)}_{L^2}^2} \\
&\quad{}-\gamma\biggbrak{R(t,\phi_0(t)e_0(x),\phi_\perp(t,x)) - 
\frac{1}{\sqrt{L}}\bigpscal{e_0}{R(t,\phi_0(t)e_0,\phi_\perp(t,\cdot))}}\6t  \\
&{}+ \frac{\sigma}{\sqrt{\eps}} \6W_\perp( t, x)\;.
\end{align}
Choosing $\alpha=\frac{1}{\sqrt{L}}$, $\beta=\sqrt{L}$ and 
$\gamma=\frac{1}{\sqrt{L}}$ yields the coupled SDE-SPDE 
system~\eqref{eq:phi0_phiperp_stoch} with 
\begin{align}
b_0(t,\phi_0,\phi_\perp)
={}& - \Bigpar{1 + \frac{1}{2L} \partial_\phi^2 b(t,\phi_0e_0)} 
\norm{\phi_\perp}_{L^2}^2 
- \frac{1}{\sqrt{L}} \pscal{e_0}{R(t,\phi_0e_0,\phi_\perp)} \;, \\
 a(t,\phi_0) 
={}& -2 \phi_0 - \frac{1}{\sqrt{L}}\partial_\phi b(t,\phi_0e_0)\;, \\
b_\perp(t,\phi_0,\phi_\perp) 
={}& -\sqrt{L}\Bigpar{1+\frac{1}{2L}\partial^2_\phi 
b(t,\phi_0e_0)}\Bigpar{\phi_\perp(\cdot)^2 - \frac1L 
\norm{\phi_\perp}_{L^2}^2}\;\\
&{}-\frac{1}{\sqrt{L}} R(t,\phi_0e_0,\phi_\perp) 
+\frac{1}{L} \pscal{e_0}{R(t,\phi_0e_0,\phi_\perp)}\;.
\label{eq:b0_a_bperp} 
\end{align}
Note that $b_0$ and $b_\perp$ are no longer local non-linearities, since they 
involve integrals over the whole torus. This remains, however, a relatively 
harmless non-locality, that will not cause any problems. 

We now derive a number of bounds on the remainder terms $b_0$ and $b_\perp$. By 
similar arguments as in the proof of Proposition~\ref{det_prop}, there exist 
constants $d, \bar d > 0$ such that whenever $\abs{\phi_0} < \bar d$ and 
$\norm{\phi_\perp}_{H^1} < d$, one has 
\begin{equation}
 \bigabs{R(t,\phi_0e_0,\phi_\perp(x))} 
 \leqs M \abs{\phi_\perp(x)}^3 
 \leqs M \CSob \norm{\phi_\perp}_{H^1}^3
\end{equation} 
for some finite constant $M$. Therefore, under these conditions on $\phi_0$ and 
$\phi_\perp$, we obtain 
\begin{equation}
\label{eq:bound_bperp_H1} 
 \bigabs{b_\perp(t,\phi_0,\phi_\perp(x))} 
 \leqs M_1 \norm{\phi_\perp}_{H^1}^3
\end{equation} 
for some constant $M_1$. Furthermore, the same argument as in 
Lemma~\ref{lemma3.3} shows that  for all 
$r<\frac12 - (2p_0+1)(\frac12 -s)$,  
there exists $C(r,s)<\infty$ such that 
\begin{equation}
 \norm{R(t,\phi_0e_0,\phi_\perp)}_{H^r} 
 \leqs C(r,s) 
\max\{\norm{\phi_\perp}_{H^s}^3,\norm{\phi_\perp}^{2p_0-1}_{H^s}\}\;. 
\end{equation} 
Combining this with the Cauchy--Schwarz inequality, we obtain the existence of 
a constant $M_2$ such that the bounds 
\begin{align}
\label{eq:bound_b0}
\bigabs{b_0(t,\phi_0,\phi_\perp)}
&\leqs M_2 
\max\bigset{\norm{\phi_\perp}^2_{H^s},\norm{\phi_\perp}^{2p_0-1}_{H^s}}\;, \\
\norm{b_\perp(t,\phi_0,\phi_\perp)}_{H^r}
&\leqs 
M_2C(r,s)\max\bigset{\norm{\phi_\perp}^2_{H^s},\norm{\phi_\perp}^{2p_0-1}_{H^s}}
\label{eq:bound_bperp}
\end{align} 
hold for all $\phi_0\in\R$ such that $\abs{\phi_0} < \bar d$.

\subsection{Deterministic case}
\label{ssec:bif_deterministic} 

We start by investigating the deterministic behaviour of the solution 
$(\phi_0(t), \phi_\perp(t,\cdot))$. The deterministic 
equation for $\phi_\perp(t,x)$ is given by 
\begin{equation}
 \6\phi_\perp(t,x) 
 = \frac{1}{\eps} \biggbrak{\Delta\phi_\perp(t,x) + a(t,\phi_0(t)) 
\phi_\perp(t,x) + b_\perp\bigpar{t,\phi_0(t),\phi_\perp(t,x)}}\6t\;.
\end{equation}

\begin{proof}[Proof of Proposition ~\ref{eq:prop_phi0_det}]
The proof is almost the same as the proof of~Proposition~\ref{det_prop}, so 
that we only comment on the differences. Here we define the Lyapunov function 
\begin{equation}
 \label{eq:Lyapunov}
 V(\phi_\perp) = \frac12 \norm{\phi_\perp}_{H^1}^2 
 = \frac12 \norm{\phi_\perp}_{L^2}^2 + \frac{L^2}{2\pi^2} 
\norm{\nabla\phi_\perp}_{L^2}^2\;. 
\end{equation} 
Its time derivative satisfies 
\begin{align}
 \eps\frac{\6}{\6t} V(\phi_\perp(t,\cdot)) 
 ={}& \pscal{\phi_\perp}{\Delta\phi_\perp} + a(t,\phi_0) \norm{\phi_\perp}_{L^2}^2 + 
\pscal{\phi_\perp}{b_\perp(t,\phi_0,\phi_\perp)}  \\
&{}- \frac{L^2}{\pi^2} \Bigbrak{\norm{\Delta\phi_\perp}_{L^2}^2 
+ a(t,\phi_0) \pscal{\Delta\phi_\perp}{\phi_\perp} 
+ \pscal{\Delta\phi_\perp}{b_\perp(t,\phi_0,\phi_\perp)}} \\
 \leqs{}& 2a(t,\phi_0) V(\phi_\perp) + 
\pscal{\phi_\perp}{b_\perp(t,\phi_0,\phi_\perp)} 
- \frac{L^2}{\pi^2}\pscal{\Delta\phi_\perp}{b_\perp(t,\phi_0,\phi_\perp)} \;.
\end{align} 
Using~\eqref{eq:bound_bperp_H1} and the Cauchy--Schwarz inequality, we obtain 
that for $\phi_0$ and $\norm{\phi_\perp}_{H^1}$ small enough, the term 
$\pscal{\phi_\perp}{b_\perp(t,\phi_0,\phi_\perp)}$ has order 
$\norm{\phi_\perp}_{H^1}^3$. As for the last term, it follows from the 
expression~\eqref{eq:b0_a_bperp} of $b_\perp$ that it has the form 
\begin{align}
 \pscal{\Delta\phi_\perp}{b_\perp(t,\phi_0,\phi_\perp)} 
 ={}& A(t)\pscal{\Delta\phi_\perp}{\phi_\perp^2}
 + B(t) \pscal{\Delta\phi_\perp}{1} \norm{\phi_\perp}_{L^2}^2 \\
 &{}- \frac{1}{\sqrt{L}} \pscal{\Delta\phi_\perp}{R(t,\phi_0e_0,\phi_\perp)}
 + \frac1L \pscal{\Delta\phi_\perp}{1} \pscal{e_0}{R(t,\phi_0e_0,\phi_\perp)}
\end{align} 
for some bounded functions $A$ and $B$. The first term on the right-hand side 
can be bounded using integration by parts. The third one has order 
$\norm{\phi_\perp}_{H^1}^2 \norm{\phi_\perp}_{L^\infty}^2$, and the other two 
terms vanish because $\pscal{\Delta\phi_\perp}{1} = 0$. It follows that 
$\pscal{\Delta\phi_\perp}{b_\perp(t,\phi_0,\phi_\perp)}$ has also order 
$\norm{\phi_\perp}_{H^1}^3$, provided $\phi_0$ and $\norm{\phi_\perp}_{H^1}$ 
are small enough.

Writing as before $\bar\tau$ for the first-exit time from the set 
$\set{V(\phi_\perp(t,\cdot))\leqs C_0}$, we obtain 
\begin{equation}
\eps \dot V 
\leqs -C_1V + C_2 V^{3/2} 
\leqs -C_1\biggbrak{1 - \frac{C_0^{1/2}C_2}{C_1}}V 
\label{borne_lyapunov_perp}
\end{equation}
for all $t \leqs \bar\tau$, and some constants $C_1, C_2 > 0$. Choosing 
$C_0$ such that $C_0^{1/2} \leqs \frac{C_1}{2C_2}$, we obtain 
\begin{equation}
 \eps \dot V \leqs -\frac12 C_1V\;,
\end{equation} 
which allows to show that there exists a particular solution satisfying $V(t) = 
0$ for all $t\in I$. 
As for $\phi_0(t)$, it obeys the ODE
\begin{equation}
\eps\dot\phi_0(t) = g(t) - \phi_0(t)^2 - b(t,\phi_0(t)e_0)\;,
\end{equation}
which can be analysed in exactly the same way as in~\cite{BG_SR},  
concluding the proof.
\end{proof} 


\subsection{Stochastic case}
\label{ssec:bif_stochastic} 

We consider now the coupled SDE--SPDE system~\eqref{eq:phi0_phiperp_stoch} with 
$\sigma>0$. We start by analysing the dynamics of $\phi_\perp(t,x)$ for a given 
realisation of $\phi_0(t)$.
The SPDE
\begin{equation}
\6\phi_\perp(t,x)= \frac{1}{\eps} \biggbrak{\Delta\phi_\perp(t,x) + 
a(t,\phi_0(t)) \phi_\perp(t,x) + b_\perp(t,\phi_0(t),\phi_\perp(t,x))}\6t + 
\frac{\sigma}{\sqrt{\eps}}\6W_\perp(t,x)\;
\end{equation}
admits, as in Subsection~\ref{sssec:stable_nonlinear}, a solution given by
\begin{align}
\phi_\perp(t,\cdot) 
={}& \frac{\sigma}{\sqrt{\eps}}\int_0^t 
\e^{\alpha(t,t_1)/\eps}\e^{\brak{(t-t_1)/\eps}\Delta}\6W(t_1,\cdot) \\
&{}+ 
\frac{1}{\eps}\int_0^t\e^{\alpha(t,t_1)/\eps}\e^{\brak{(t-t_1)/\eps}\Delta}
b_\perp(t_1,\phi_0(t_1),\phi_\perp(t_1,\cdot))\6t_1\;,
\end{align} 
where $\alpha(t,t_1) = \int_{t_1}^t a(u,\phi_0(u))\6u$. 

\begin{proof}[Proof of Theorem~\ref{thm:phiperp}]
The proof is virtually the same as the proof of Theorem~\ref{thm:stable}, the 
only difference being that we use here the fact that $\phi_0(t)$ is bounded by 
a constant of order $T_0$, owing to the definition of $\cB_0(h)$. Therefore, 
$a(t,\phi_0)$ is bounded above by a constant of order $T_0$. Since the largest 
eigenvalue of the Laplacian acting on mean-zero functions $\phi_\perp$ is 
equal to $-\pi^2/L^2$, taking $T_0$ small enough we obtain again a
bound of the form~\eqref{borne_lyapunov_1} for the Lyapunov function $V = 
\norm{\phi_\perp}_{H^1}^2$.  
\end{proof}

We now fix a realisation of $\phi_\perp(t)$. The difference 
$\psi_0(t) = \phi_0(t) - \bar\phi_0(t)$ satisfies the SDE 
\begin{equation}
\label{eq:psi0_stoch} 
 \6\psi_0(t)=\frac{1}{\eps}\Bigbrak{\bar a\bigpar{t,\bar\phi_0(t)}\psi_0(t) 
 + \bar b(t,\psi_0(t))}\6t+\frac{\sigma}{\sqrt{\eps}}\6W_0(t)\;,
\end{equation} 
where 
\begin{equation}
\bar a(t,\bar\phi_0)=-2\bar\phi_0
-\partial_{\phi_0}b(t,\bar\phi_0 e_0)\;,
\end{equation} 
and $\bar b(t,\psi_0(t))$ denotes a non-linear term given by
\begin{align}
 \bar b(t,\psi_0)
 ={}& -\Bigpar{1 + 
\frac{1}{2}\partial_{\phi_0}b(t,\bar\phi_0e_0+\theta\psi_0e_0)}
\psi_0^2 \\ 
&{} + b_0(t,\bar\phi_0(t) + \psi_0,\phi_\perp(t,\cdot))
- b_0(t,\bar\phi_0(t),\phi_\perp(t,\cdot))
\end{align} 
for some $\theta\in (0, 1)$. By~\eqref{eq:bound_b0}, there is a 
constant $M>0$ such that $\bar b(t,\psi_0(t))$ satisfies 
\begin{equation}
\bigabs{\bar b(t,\psi_0(t))}\leqs M\psi_0(t)^2
+ 2M_2\|\phi_\perp\|^2_{H^s}\leqs 
M\psi_0(t)^2
+ 2M_2 h_\perp^2 \qquad \forall t< \tau_{\cB_\perp(h_\perp)}\;.
\end{equation} 
A solution of \eqref{eq:psi0_stoch} is given by 
$\psi_0(t)=\psi_0^0(t)+\psi_0^1(t)$,where $\psi_0^0(t)$ is the solution of the 
linearisation  of \eqref{eq:psi0_stoch}, and 
\begin{equation}
\label{eq:psi0^1} 
\psi_0^1(t)=\frac{1}{\eps}\int^t_{-T_0}\e^{\bar\alpha(t,t_1)/\varepsilon}\bar 
b(t_1,\psi_0(t_1))\6t_1\;,
\end{equation} 
where $\bar\alpha(t,t_1) = \int_{t_1}^t \bar a(t_2,\bar\phi_0(t_2))\6t_2$.

Recall that we introduced a variance-related function $\zeta(t)$ 
satisfying~\eqref{eq:zeta}. 
According to~\cite[Proposition 3.8]{BG_SR},  
$\psi_0^0(t)$ is likely to remain in a strip of width proportional to 
$\sqrt{\zeta(t)}$. More precisely, 
\begin{equation}
\label{eq:prob_psi0^0} 
\biggprob{\sup_{-T_0\leqs t_1\leqs t} 
\frac{\abs{\psi^0_0(t_1)}}{\sqrt{\zeta(t_1)}}\geqs h}
\leqs C(t,\eps) \exp\biggset{-\frac{h^2}{2\sigma^2}(1-\Order{\eps})}\;,
\end{equation} 
where 
\begin{equation}
C(t,\eps) = \frac{\abs{\bar\alpha(t,-T_0)}}{\eps^2} + 2 \;.
\end{equation} 
We now use this estimate to prove Theorem~\ref{thm:phi0_stable}.

\begin{proof}[Proof of Theorem~\ref{thm:phi0_stable}]
For any decomposition $h=h_0+h_1$ with $h_0, h_1>0$, one has
\begin{align}
&\bigprob{\tau_{\cB_0(h)} < t\wedge\tau_{\cB_\perp(h_\perp)}}
= \bigprob{\tau_{\cB_0(h)} < t, \tau_{\cB_0(h)} < \tau_{\cB_\perp(h_\perp)}}\;\\
&\qquad\leqs \biggprob{\sup_{-T_0\leqs t_1\leqs 
t\wedge\tau_{\cB_0(h)}}\frac{\abs{\psi_0(t_1)}}{\sqrt{\zeta(t_1)}}\geqs h, 
\tau_{\cB_0(h)}<\tau_{\cB_\perp(h_\perp)}}\;\\
&\qquad\leqs \biggprob{\sup_{-T_0\leqs t_1\leqs 
t}\frac{\abs{\psi^0_0(t_1)}}{\sqrt{\zeta(t_1)}}\geqs h_0} + 
\biggprob{\sup_{-T_0\leqs t_1\leqs 
t\wedge\tau_{\cB_0(h)}}\frac{\abs{\psi^1_0(t_1)}}{\sqrt{\zeta(t_1)}}\geqs h_1, 
\tau_{\cB_0(h)}<\tau_{\cB_\perp(h_\perp)}} \;
\end{align} 
The first probability satisfies the bound \eqref{eq:prob_psi0^0}, so that it 
remains to control the second one. By \eqref{eq:psi0^1} and for all 
$t_1 \leqs t\wedge\tau_{\cB_0(h)}< \tau_{\cB_\perp(h_\perp)}$, as 
in~\cite[Proposition~3.10]{BG_SR}, we have the bound 
\begin{align}
\frac{\abs{\psi^1_0(t_1)}}{\sqrt{\zeta(t_1)}}&\leqs
\frac{(Mh^2\zeta(t_1) + M_2h_\perp^2)}{\sqrt{\zeta(t_1)}}
\frac{1}{\eps}\int^{t_1}_{-T_0}\e^{\bar\alpha(t_1,t_2)/\eps}
\6t_2\;\\
&\leqs Mh^2\hat\zeta(t)^{3/2} + M_2h_\perp^2\hat\zeta(t)^{1/2}\;.
\end{align} 
Choosing $h_\perp^2\leqs \frac{M}{M_2}h^2\hat\zeta(t)$ and $h_1=const\ 
h^2\hat{\zeta}(t)^{3/2}$, we get
 \begin{equation}
\biggprob{\sup_{-T_0\leqs t_1\leqs 
t\wedge\tau_{\cB_0(h)}}\frac{\abs{\psi^1_0(t_1)}}{\sqrt{\zeta(t_1)}}\geqs h_1, 
\tau_{\cB_0(h)}<\tau_{\cB_\perp(h_\perp)}}=0\;.
\end{equation}
Therefore,
\begin{equation}
\bigprob{\tau_{\cB_0(h)} < t\wedge\tau_{\cB_\perp(h_\perp)}}\leqs C(t,\eps) 
\exp\biggset{-\frac{h_0^2}{2\sigma^2}(1-\Order{\eps})}\;.
\end{equation} 
We thus obtain the result by choosing 
$h_0 = h-h_1
= h - \Order{h^2\hat{\zeta}(t)^{3/2}}
= h(1-\Order{h\hat{\zeta}(t)^{3/2}})$.
\end{proof} 

In weak noise regime, the probability of leaving either $\cB_0(h)$ or 
$\cB_\perp(h_\perp)$ before time $t$ is given by 
\begin{align}
&\bigprob{\tau_{\cB_0(h)}\wedge\tau_{\cB_\perp(h_\perp)}<t}\;\\
&\;\;= 
\bigprob{\tau_{\cB_0(h)}\wedge\tau_{\cB_\perp(h_\perp)}<t,\tau_{\cB_0(h)}<\tau_{
\cB_\perp(h_\perp)}} 
+\bigprob{\tau_{\cB_0(h)}\wedge\tau_{\cB_\perp(h_\perp)}<t,\tau_{
\cB_\perp(h_\perp)}\leqs\tau_{\cB_0(h)}}\;\\
&\;\;= \bigprob{\tau_{\cB_0(h)}<t,\tau_{\cB_0(h)}<\tau_{\cB_\perp(h_\perp)}} 
+\bigprob{\tau_{\cB_\perp(h_\perp)}<t,\tau_{\cB_\perp(h_\perp)}\leqs\tau_{
\cB_0(h)}}\\
&\;\;= \bigprob{\tau_{\cB_0(h)}<t \wedge \tau_{\cB_\perp(h_\perp)}} 
+\bigprob{\tau_{\cB_\perp(h_\perp)} < t \wedge \tau_{\cB_0(h)}}\;.
\end{align} 
The first probability on the right-hand side is bounded by 
Theorem~\ref{thm:phi0_stable} and the second one by Theorem~\ref{thm:phiperp}. 
Thus, we conclude that the behaviour of $\phi_0(t)$ in this regime does not 
differ much from the behaviour of the deterministic solution $\bar \phi_0(t)$ 
during the whole time interval $[-T_0,T_0]$. 

However, in the strong-noise regime, the situation is different. We assume from 
now on that $\sigma\geqs(\eps\vee\delta)^{3/4}$, where 
Theorem~\ref{thm:phi0_stable} shows that sample paths are concentrated near the 
adiabatic solution tracking the stable potential well at $\phi^*_+$ up to times 
of order $-\sigma^{2/3}$. As time increases, it quickly becomes very unlikely 
not to reach and overcome the unstable solution $\smash{\hat\phi_0}(t)$ tracking 
$\phi^*_-$. We notice that the linearisation of $f$ at $\smash{\hat\phi_0}$ 
satisfies  
\begin{equation}
\label{eq:ahat_zeta} 
\hat a(t,\hat\phi_0(t))  \asymp  \bigpar{\abs{t} \vee \sqrt{\delta \vee 
\eps}\,}  
\asymp \bigabs{\bar a(t,\bar\phi_0(t))}  \asymp \frac{1}{\zeta(t)}\;.
\end{equation} 
In what follows, we prove Theorem~\ref{thm:phi0_strong}, where the two terms on 
the right-hand side of \eqref{eq:prob_phi0_strong} bound, respectively, the 
probability that $\phi_0$ does not reach $-d$ before time $t$, while staying 
below $\bar\phi_0 + h\sqrt{\zeta}$, and the probability that $\phi_0$ 
crosses the level $\bar\phi_0 + h\sqrt{\zeta}$ before time t. 

\begin{proof}[Proof of Theorem~\ref{thm:phi0_strong}] 
Let $h$ be such that $\bar\phi_0(t) + h\sqrt{\zeta(t)}\leqs d$ for all 
$t\in[-c_1\sigma^{2/3},c_1\sigma^{2/3}]$. We introduce the stopping times
\begin{align}
\tau_+&=\inf\biggset{t_1\in\brak{-c_1\sigma^{2/3},T_0} \colon
\frac{\phi_0(t_1)-\bar\phi_0(t_1)}
{\sqrt{\zeta(t_1)}}>h}\;,\\
\tau_-&=\inf\bigset{t_1\in\brak{-c_1\sigma^{2/3},T_0}\colon \phi_0(t_1)<-d}.
\end{align} 
Then, the probability that $\phi_0$ does not reach $-d$ while $\phi_\perp$ 
remains in $\cB_\perp(h_\perp)$ is given by
\begin{align}
\prob{\tau_- &> t \wedge \tau_{\cB_\perp(h_\perp)}}\;\\
&=\prob{\tau_- > t \wedge \tau_{\cB_\perp(h_\perp)}, \tau_+ \leqs t \wedge 
\tau_{\cB_\perp(h_\perp)}} + 
\prob{\tau_- > t \wedge \tau_{\cB_\perp(h_\perp)}, \tau_+ > t \wedge 
\tau_{\cB_\perp(h_\perp)}}\;\\
&\leqs\prob{\tau_+ \leqs t \wedge 
\tau_{\cB_\perp(h_\perp)}} + \prob{\tau_- \wedge \tau_+ > t 
\wedge \tau_{\cB_\perp(h_\perp)}}\;
\label{eq:decomp_tauplusminus} 
\end{align} 
We estimate these two terms separately and the crucial term is the second one. 
Since we are going to use the Markov property and restart the process at certain 
times, we will use the notation $\fP^{t_0,\phi_0}$ for the law of the process 
started at time $t_0$ in $\phi_0$ whenever necessary. 

\begin{prop}
Under the assumptions of Theorem~\ref{thm:phi0_strong}, there exist constants 
$\kappa_1, M_3 > 0$ such that whenever 
$(-c_1\sigma^{2/3},\phi_{0,0})\in\cB_0(h/2)$, 
one has 
\begin{equation}
\probin{-c_1\sigma^{2/3},\phi_{0,0}}{\tau_+\leqs t\wedge 
\tau_{\cB_\perp(h_\perp)}}
\leqs C(t,\eps) 
\exp\biggset{-\frac{\kappa_1}{2\sigma^2} 
\biggpar{h - M_3h_\perp^2\sqrt{\hat\zeta(t)}\,}^2
}\;,
\end{equation}
for all $t\in[-c_1\sigma^{2/3}, T_0]$, where $C(t,\eps)= 
\frac{\abs{\bar\alpha(t,-c_1\sigma^{2/3})}}{\eps^2} + 2$.
\end{prop} 

\begin{proof}
The solution of \eqref{eq:psi0_stoch} is given by   
\begin{equation}
\psi_0(t)=\psi_0^0(t)+\frac{1}{\eps}\int^t_{-c_1\sigma^{2/3}}\e^{\bar\alpha(t,
t_1)/\varepsilon}\bar b(t_1,\psi_0(t_1))\6t_1\;.
\end{equation} 
We define a partition $-c_1\sigma^{2/3}=u_0<u_1<\dots<u_K=t$ of 
$[-c_1\sigma^{2/3},t]$   
by
\begin{equation}
\bar\alpha(u_k,u_{k-1})=\eps \qquad \text{for } 1\leqs k\leqs 
K=\bigg\lceil 
\frac{\bar{\alpha}(t,-c_1\sigma^{2/3})}{\eps}\bigg\rceil\;.
\end{equation} 
We also introduce the notation $\rho_k = \frac12 h \sqrt{\zeta(u_k)}$. 
As shown in~\cite[Proposition~3.12]{BG_SR}, the Markov property implies 
\begin{align}
 \probin{-c_1\sigma^{2/3},\phi_{0,0}}{\tau_+< t\wedge \tau_{\cB_\perp(h_\perp)}}
 &= \biggprobin{-c_1\sigma^{2/3},\phi_{0,0}}{\sup_{-c_1\sigma^{2/3}\leqs 
t_1\leqs 
 t\wedge \tau_{\cB_\perp(h_\perp)}}
\frac{\psi_0(t_1)}{\sqrt{\zeta(t_1)}} >h} \\
 &\leqs \sum_{k=0}^{K-1} Q_k\;,
\end{align} 
where 
\begin{align}
 Q_k = 
 \sup_{\psi_0(u_k) \leqs \rho_k} 
 \biggl[ & \biggprobin{u_k,\psi_0(u_k)}{\sup_{u_k\leqs t_1\leqs u_{k+1}}
\frac{\psi_0(t_1)}{\sqrt{\zeta(t_1)}} >h} \\
&{}+ \biggprobin{u_k,\psi_0(u_k)}{\sup_{u_k\leqs t_1\leqs u_{k+1}}
\frac{\psi_0(t_1)}{\sqrt{\zeta(t_1)}} \leqs h,\; 
\psi_0(u_{k+1}) > \rho_{k+1})} \biggr]\;.
\end{align}
For $h$ smaller than a constant of order $1$ and $t_1 \leqs 
\tau_{\cB_\perp(h_\perp)}$, \eqref{eq:bound_b0} shows that 
$\bar b(t_1,\psi_0(t_1))$ is bounded by 
$M_2 h_\perp^2$. It follows that for any $t_1 \in [u_k,u_{k+1}]$, one has 
\begin{align}
\psi_0^1(t_1) 
&\leqs 
M_2 h_\perp^2 \int_{-c_1\sigma^{2/3}}^{t} \frac{1}{-\bar 
a(t_1,\bar\phi_0(t_1))}
\frac{-\bar a(t_1,\bar\phi_0(t_1))}{\eps} \e^{\bar\alpha(t,t_1)/\eps} \6t_1 \\
&\leqs 
M_2 h_\perp^2 \sup_{u\in[u_k,u_{k+1}]}\frac{1}{\abs{\bar a(u,\bar\phi_0(u))}}\;.
\label{eq:bound_psi0t1} 
\end{align} 
Therefore, there is a constant $M_3$ such that for any $t_1 \in[u_k,u_{k+1}]$ 
one has 
\begin{equation}
 \frac{\psi_0^1(t_1)}{\sqrt{\zeta(t_1)}}
 \leqs M_3 h_\perp^2 \sqrt{\hat{\zeta}(u_{k+1})}\;.
\end{equation} 
Proceeding as in the proof of~\cite[Proposition~3.12]{BG_SR}, but with a 
shifted value of $h$, one obtains 
\begin{equation}
 P_k \leqs \exp\biggset{-\frac{\kappa_1}{\sigma^2} 
 \biggpar{h - M_3 h_\perp^2 \sqrt{\hat{\zeta}(u_{k+1})}}^2\,}
\end{equation} 
for some $\kappa_1 > 0$, which implies the claimed result. 
\end{proof} 

The main part of the proof is contained in the following estimate, whose proof 
is very close in spirit to the proof of~\cite[Proposition~4.6]{BG_SR}, but with 
some changes due to the zero-mean part $\phi_\perp$ of the field.   

\begin{prop}
Under the assumptions of Theorem~\ref{thm:phi0_strong}, there exists a choice of 
$c_1>0$ and constants $\bar c_\perp$ and $\kappa_2>0$ such that for $0 < h_\perp 
< \bar c_\perp \sigma^{2/3}$, and all initial conditions $\phi_{0,0}$ in the 
interval $(-d, \bar\phi_0(-c_1\sigma^{2/3}) + 
h\sqrt{\zeta(-c_1\sigma^{2/3})}\,]$,  one has  
\begin{align}
&\probin{-c_1\sigma^{2/3},\phi_{0,0}}{\tau_- \wedge \tau_+ > t 
\wedge \tau_{\cB_\perp(h_\perp)}}\\
&\qquad = \Bigprobin{-c_1\sigma^{2/3},\phi_{0,0}}
{-d<\phi_0(t_1)\leqs\bar\phi_0(t_1)+h\sqrt{\zeta(t_1)} \;
\forall t_1\in[-c_1\sigma^{2/3},t\wedge\tau_{\cB_\perp(h_\perp)}]}\;\\
&\qquad \leqs \frac32 
\exp\biggset{-\kappa_2\frac{\hat\alpha(t,-c_1\sigma^{2/3})}{\log(\sigma^{-1}
)\eps } } \;.
\end{align} 
\end{prop} 

\begin{proof}
Let $\varrho\geqs 1$ and define a 
partition $-c_1\sigma^{2/3}=u_0<u_1<\dots<u_K=t$ 
by
\begin{equation}
\hat\alpha(u_k,u_{k-1})=\varrho\eps \qquad \text{for } 1\leqs k\leqs 
K=\bigg\lceil 
\frac{\hat{\alpha}(t,-c_1\sigma^{2/3})}{\varrho\varepsilon}\bigg\rceil\;.
\label{eq:def_K} 
\end{equation} 
Writing
\begin{equation}
Q_k=\sup_{\phi_0(u_k)\in (-d, \bar\phi_0(u_k)+h\sqrt{\zeta(u_k)}]}
\Bigprobin{u_k,\phi_0(u_k)}{
-d<\phi_0 \leqs \bar\phi_0(t_1)+h\sqrt{\zeta(t_1)}\ \forall 
t_1\in 
[u_k,u_{k+1}]}\;,
\end{equation} 
we have, by the Markov property, 
\begin{align}
&\bigprobin{-c_1\sigma^{2/3},\phi_{0,0}}
{-d<\phi_0(t_1)\leqs\bar\phi_0(t_1)+h\sqrt{
\zeta(t_1) } 
\;\forall t_1\in[-c_1\sigma^{2/3},t]}\;\\
&\qquad = \Bigexpecin{-c_1\sigma^{2/3},\phi_{0,0}}{1_{\{
-d<\phi_0(t_1)\leqs\bar\phi_0(t_1)+h\sqrt{
\zeta(t_1)} \;\forall t_1\in [-c_1\sigma^{2/3},u_{K-1}]\}}\;\\
&\qquad\times\bigprobin{u_{K-1},\phi_0(u_{K-1})}{
-d<\phi_0(t_1)\leqs\bar\phi_0(t_1)+h\sqrt{\zeta(t_1)} \;\forall t_1\in 
[u_{K-1},u_K]}}\;\\
&\qquad \leqs 
Q_{K-1}\bigprobin{-c_1\sigma^{2/3},\phi_{0,0}}{
-d<\phi_0(t_1)\leqs\bar\phi_0(t_1)+h\sqrt{
\zeta(t_1) } 
\;\forall t_1\in[-c_1\sigma^{2/3},u_{K-1}]}\;\\
&\qquad \leqs \dots \leqs \prod_{k=0}^{K-1}Q_k.
\end{align}
Our plan is to show that for an appropriate choice of $\varrho$, $Q_k$ is 
bounded away from $1$ for $k=0,\dots,K-1$. In order to estimate $Q_k$ we shall 
distinguish three cases corresponding to $\phi_0$ crossing the levels 
$\bar\phi_0$ and $\hat\phi_0$ before reaching $-d$. We set 
\begin{equation}
 \overline M_k = M_2 h_\perp^2 \sup_{u\in[u_k,u_{k+1}]}\frac{1}{\abs{\bar 
a(u,\bar\phi_0(u))}}\;, 
\qquad 
 \widehat M_k = M_2 h_\perp^2 \sup_{u\in[u_k,u_{k+1}]}\frac{1}{ 
\hat a(u,\hat\phi_0(u))}\;,
\end{equation} 
and introduce  a further subdivision $u_k<\tilde u_{k,1}<\tilde 
u_{k,2}<u_{k+1}$ defined by
\begin{equation}
\hat\alpha(\tilde u_{k,1},u_k)=\frac{1}{3}\varrho\eps\;,
\qquad \hat\alpha(\tilde u_{k,2},u_k)=\frac{2}{3}\varrho\eps\;.
\end{equation} 
Define the stopping times
\begin{align}
\tau_{k,1}&=\inf\biggset{t_1\in[u_k,\tilde u_{k,1}]: 
\phi_0(t_1)\leqs\bar\phi_0(t_1) + \overline M_k}\;,\\
\tau_{k,2}&=\inf\biggset{t_1\in[u_k,\tilde 
u_{k,2}]:\phi_0(t_1)\leqs\hat\phi_0(t_1) + \widehat M_k}\;.
\end{align}
Then we can write
\begin{align}
&\probin{u_k,\phi_0(u_k)}{
-d<\phi_0(t_1) \leqs \bar\phi_0(t_1)+h\sqrt{\zeta(t_1)}\; \forall 
t_1\in [u_k,u_{k+1}]}\;\\
&\leqs 
\biggprobin{u_k,\phi_0(u_k)}{\bar\phi_0(t_1) + \overline
M_k<\phi_0(t_1)\leqs\bar\phi_0(t_1)+h\sqrt{
\zeta(t_1) } \;
\forall t_1\in [u_k,\tilde u_{k,1}]}\;\\
&\qquad{}+ \Bigexpecin{u_k,\phi_0(u_k)}{1_{\{\tau_{k,1}<\tilde 
u_{k,1}\}} \\
&\qquad\quad \times 
\probin{\tau_{k,1},\phi_0(\tau_{k,1})}{-d<\phi_0
(t_1)\leqs\bar\phi_0(t_1) +h\sqrt{
\zeta(t_1) } \; \forall t_1\in 
[\tau_{k,1},u_{k+1}]}}\;.
\label{eq:decomp_3steps_1} 
\end{align}
We start by bounding the first term on the right-hand side. 
Let  
\begin{equation}
 \psi_0^{(k)}(t_1)
 = \psi_0(u_k) \e^{\bar\alpha(t_1,u_k)/\eps} 
 + \frac{\sigma}{\sqrt{\eps}} \int_{u_k}^{t_1} \e^{\bar\alpha(t_1,v)/\eps} 
\6W_0(v)
\end{equation} 
be the solution of the equation linearised around $\bar\phi_0(t)$, starting in 
$\psi_0(u_k) = \phi_0(u_k) - \bar\phi_0(u_k)$. Then in follows 
from~\eqref{eq:bound_psi0t1} that 
\begin{equation}
 \psi_0(t) \leqs 
 \psi_0^{(k)}(t)+\overline M_k 
 \qquad 
 \forall t_1 \in [u_k,u_{k+1}]\;.
\end{equation} 
Note that $\psi_0^{(k)}(\tilde u_{k,1})$ is a normal random variable with 
parameters  
\begin{align}
\expec{\psi_0^{(k)}(\tilde u_{k,1})}
&=\psi_0(u_k)\e^{\bar\alpha(\tilde u_{k,1},u_k)/\eps}
\leqs \psi_0(u_k)\e^{-\varrho/3R}\;\\
\variance(\psi_0^{(k)}(\tilde{u}_{k,1}))
&= \frac{\sigma^2}{\eps} \int_{u_k}^{\tilde u_{k,1}} 
\e^{2\bar\alpha(t_1,v)/\eps} \6v \\
&\geqs \frac{\sigma^2}{2}
\inf_{u_k\leqs t_1\leqs u_{k+1}} \frac{1}{\abs{\bar a(u,\bar\phi_0(u))}}
[1-\e^{-2\varrho/3R}]\;,
\end{align}
where $R>0$ is a constant such that $\hat a(t_1,\hat\phi_0(t_1))\leqs 
R\abs{\bar 
a(t_1,\bar\phi_0(t_1))}$ for all $t_1\in[-c_1\sigma^{2/3},t]$. 
Then Andr\'e's reflection principle shows that the first term on the right-hand 
side of~\eqref{eq:decomp_3steps_1} is bounded above by 
\begin{align}
 \bigprobin{u_k,\phi_0(u_k)}{\psi_0^{(k)}(t_1) > 0 \; 
 \forall t_1\in[u_k,\tilde u_{k,1}] }
 &= 1 - 2 \bigprobin{u_k,\phi_0(u_k)}{\psi_0^{(k)}(t_1) \leqs 0} \\ 
 &= 2 \bigprobin{u_k,\phi_0(u_k)}{\psi_0^{(k)}(t_1) > 0} - 1 \\
 &\leqs \frac{2}{\sqrt{\pi}}\frac{h}{\sigma}
 C_1(k) \frac{\e^{-\varrho/3R}}{\sqrt{1-\e^{-2\varrho/3R} }}\;,
\label{eq:borne_phi0_bar}
\end{align}
where 
\begin{equation}
 C_1(k) = \sup_{u_k\leqs t_1\leqs u_{k+1}} 
 \sqrt{\abs{\bar a(u_k,\bar\phi_0)}}\sqrt{\zeta(u_k)}
\end{equation} 
is a constant of order $1$, owing to~\eqref{eq:ahat_zeta}. 

In order to bound the second term on the right-hand side 
of~\eqref{eq:decomp_3steps_1}, we set set 
$\varphi_0(t)=\phi_0(t)-\hat\phi_0(t)$, 
where we recall that $\hat\phi_0(t)$ is the deterministic solution tracking 
$\phi^*_-$. Observe that if $\tau_{k,1}<\tilde u_{k,1}$, we also have 
\begin{align}
\label{eq:prob_tau_k,1}
&\Bigprobin{\tau_{k,1},\varphi_0(\tau_{k,1})}{-d < \hat\phi_0(t_1) + 
\varphi_0(t_1) 
\leqs\bar\phi_0(t_1)+h\sqrt{\zeta(t_1)} \; \forall t_1\in 
[\tau_{k,1},u_{k+1}]}\;\\
&\leqs 
\Bigprobin{\tau_{k,1},\varphi_0(\tau_{k,1})}{
\widehat M_k<\varphi_0(t_1)\leqs\bar\phi_0(t_1)-\hat\phi_0(t_1)+h\sqrt{
\zeta(t_1) } 
\;\forall t_1\in [\tau_{k,1},\tilde u_{k,2}]}\;\\
&\qquad{}+\Bigexpecin{u_k,\varphi_0(u_k)}{1_{\{\tau_{k,2}<\tilde u_{k,2}\}} \\
&\qquad{}\quad\times \probin{\tau_{k,2},\varphi_0(\tau_{k,2})}{-d<\hat{\phi_0}
(t_1)+\varphi_0(t_1)\leqs\bar\phi_0(t_1)+h\sqrt{\zeta(t_1)} \;\forall t_1\in 
[\tau_{k,2},u_{k+1}]}}\;.
\end{align}
To bound the first term on the right-hand side, we introduce the linear process 
\begin{equation}
 \varphi_0^{(k)}(t_1)
 = \varphi_0(\tilde u_{k,1}) \e^{\hat\alpha(t_1,\tilde u_{k,1})/\eps} 
 + \frac{\sigma}{\sqrt{\eps}} \int_{\tilde u_{k,1}}^{t_1} 
\e^{\hat\alpha(t_1,v)/\eps} \6W_0(v)
\end{equation} 
which satisfies 
\begin{equation}
 \varphi_0(t) \leqs 
 \varphi_0^{(k)}(t)+\widehat M_k 
 \qquad 
 \forall t_1 \in [\tilde u_{k,1},u_{k+1}]\;.
\end{equation} 
Then we have the estimates
\begin{align}
\expec{\varphi_0^{(k)}(\tilde 
u_{k,2})}&=\varphi_0(\tau_{k,1})\e^{\hat\alpha(\tilde 
u_{k,2},\tau_{k,1})/\eps} \\
&\leqs \bigbrak{\bar\phi_0(\tau_{k,1}) + \overline M_k - \hat\phi_0(\tau_{k,1})}
\e^{\varrho/3}\;,\\
\e^{-2\hat\alpha(\tilde 
u_{k,2},\tau_{k,1})/\eps}\variance(\varphi_0^{(k)}(\tilde u_{k,2}))
&\geqs \inf_{u_k\leqs t_1\leqs u_{k+1}}\frac{\sigma^2}{2\hat 
a(t_1,\hat\phi_0(t_1))}[1-\e^{-2\varrho/3}]\;.
\end{align}
The first term on the right-hand side of~\eqref{eq:prob_tau_k,1} can then be 
bounded by 
\begin{equation}
\label{eq:borne_phi0_hat}
 \frac{2}{\sqrt{\pi}}\frac{1}{\sigma}C_2(k)\frac{1}{\sqrt{
1-\e^{-2\varrho/3}}}\;,
\end{equation} 
where
\begin{equation}
 C_2(k) = \sup_{u_k\leqs t_1\leqs u_{k+1}}\sqrt{\hat 
a(t_1,\hat\phi_0(t_1))}
\sup_{u_k\leqs t_1\leqs u_{k+1}}(\bar\phi_0(t_1) + \overline M_k - 
\hat{\phi}_0(t_1))\;.
\end{equation} 
Finally, in order to estimate the second summand in \eqref{eq:prob_tau_k,1}, 
we use the end point estimate
\begin{align}
&\bigprobin{\tau_{k,2},\varphi_0(\tau_{k,2})}{-d<\hat{\phi_0}
(t_1)+\varphi_0(t_1)\leqs\bar\phi_0(t_1)+h\sqrt{\zeta(t_1)} \;\forall t_1\in 
[\tau_{k,2},u_{k+1}]}\\
&\qquad\leqs \bigprobin{\tau_{k,2},\varphi_0(\tau_{k,2})}{-d<\hat{\phi_0}
(u_{k+1})+\varphi_0(u_{k+1})} \\
&\qquad\leqs \frac{1}{2}+\frac{1}{\sqrt{\pi}}\frac{1}{\sigma} C_3(k)
\biggbrak{d+\hat\phi_0(u_{k+1})+\widehat M_k}
\frac{\e^{-\varrho/3}}{\sqrt{1-\e^{-2\varrho/3}}}\;,
\label{eq:borne_endpoint}
\end{align}
where 
\begin{equation}
 C_3(k) = \sup_{u_k\leqs t_1\leqs u_{k+1}}\sqrt{\hat a(t_1,\hat\phi_0(t_1))}\;.
\end{equation}
Summing~\eqref{eq:borne_phi0_bar}, \eqref{eq:borne_phi0_hat} 
and~\eqref{eq:borne_endpoint} we get the existence of a constant $C_0>0$ such 
that 
\begin{equation}
Q_k\leqs\frac{1}{2}+C_0\biggbrak{\frac{h}{\sigma}\e^{-\varrho/3R} 
C_1(k) +\frac{1}{\sigma}C_2(k) 
 +\frac{1}{\sigma}\e^{-\varrho/3}C_3(k) \bigpar{1+\widehat M_k}}\;.
\end{equation}
Since 
\begin{align}
|t_1|&\leqs c_1\sigma^{2/3}\;,\\
\hat a(t_1,\hat\phi_0(t_1))&\asymp |t_1|\vee \sqrt{\delta\vee\eps}\;,\\
\bar\phi_0(t_1)-\hat\phi_0(t_1)&\leqs \bar c_1|t_1|\;,
\end{align}
where $\bar c_1$ is proportional to $c_1$ 
and $\varrho\geqs1$, there exists another constant $C_4$ such that 
\begin{equation}
Q_k\leqs 
\frac{1}{2}+C_4\Bigbrak{\frac{h}{\sigma}\e^{-\varrho/3R} 
+ \bar c_1^{\,3/2} + \frac{h_\perp^2}{\sigma^{4/3}} (1 + \e^{-\varrho/3})
+ \frac{1}{\sigma^{2/3}} 
\e^{-\varrho/3}}\;.
\end{equation}
Choosing $h_\perp\leqs \bar c_\perp\sigma^{2/3}$ we get
\begin{equation}
Q_k\leqs \frac{1}{2} + 
C_4\Bigbrak{\frac{h}{\sigma}\e^{-\varrho/3R} + \bar c_1^{\,3/2} 
+ 2\bar c_\perp^2 + \frac{1}{\sigma^{2/3}} \e^ {-\varrho/3}}\;.
\end{equation}
For $c_1$ such that $\bar c_1^{\,3/2} = 2\bar c_\perp^2 = \frac{1}{24C_4}$ and 
\begin{equation}
 \varrho = 3R\log \Bigpar{36C_4\frac{h}{\sigma}} 
 \vee 3 \log \Bigpar{\frac{18C_4}{c_1\sigma^{2/3}}} 
 \vee 1\;, 
\end{equation} 
$Q_k$ is bounded by $\frac{2}{3}$ for 
$k=0, \dots, K-1$.
We conclude that with this choice of $\varrho$, we have 
\begin{align}
\Bigprobin{-c_1\sigma^{2/3},\phi_0}{&-d<\phi_0(t_1)\leqs\bar\phi_0(t_1)+h\sqrt{
\zeta(t_1) } 
\forall t_1\in[-c_1\sigma^{2/3},t]} \\
&\leqs \biggpar{\frac{2}{3}}^{K-1} 
=\frac{3}{2}\exp\biggset{-K\log\biggpar{\frac32}}\;
\end{align}
which yields the claimed result, owing to our choice~\eqref{eq:def_K} of $K$, 
and the fact that $\varrho$ has order $\log(\sigma^{-1})$. 
\end{proof} 

The conclusion of Theorem~\ref{thm:phi0_strong} now follows immediately by 
combining the last two propositions. 
\end{proof}

\begin{proof}[Proof of Proposition~\ref{prop:d0}]
We introduce the stopping times
\begin{align}
\tau_+&=\inf\bigset{t_1\in[t_0,t_0 + \tilde c\eps] \colon
\phi_0(t_1) > -d+\rho}\;,\\
\tau_-&=\inf\bigset{t_1\in[t_0,t_0 + \tilde c\eps] \colon
\phi_0(t_1) < -d_0}\;,
\end{align} 
and the process 
\begin{equation}
 \tilde\phi_0(t) = -d - \frac{1}{\eps}f_0(t-t_0)
 + \frac{\sigma}{\sqrt{\eps}} W_{t-t_0}\;.
\end{equation}
Taking the constant $M$ in the statement of the proposition equal to $M_2$, one 
can check, in a similar way as before, that $\phi_0(t_1) \leqs 
\tilde\phi_0(t_1)$ for all $t_1 \leqs \tau_-\wedge\tau_+\wedge 
\tau_{\cB_\perp(h_\perp)}$. Now we observe that 
\begin{align}
\bigprob{\tilde\phi_0(t) &\geqs -d_0 \;\forall t\in[t_0,t_0 + \tilde c\eps]} \\
\leqs{}& \biggprob{\sup_{t\in[t_0,t_0 + \tilde c\eps]} 
\Bigbrak{\tilde\phi_0(t) + \frac{1}{\eps}f_0(t-t_0)} > -d + \rho} \\
&{}+ \Bigprob{-d_0 \leqs \tilde\phi_0(t) \leqs -d+\rho - 
\frac{1}{\eps}f_0(t-t_0)\;\forall t\in[t_0,t_0 + \tilde c\eps]}\;.
\end{align}
The second term on the right-hand side vanishes as soon as we take $\tilde c > 
(d_0-d+\rho)/f_0$, while the first one is equal to 
\begin{equation}
 \biggprob{\sup_{t\in[t_0,t_0 + \tilde c\eps]} \frac{\sigma}{\sqrt{\eps}} 
W_{t-t_0} > \rho} 
\leqs \e^{-\rho^2/(2\tilde c\sigma^2)}
\end{equation} 
by a Bernstein-type inequality. Now we note that for any $t\in[t_0, t_0 + 
\tilde c\eps]$, we have 
\begin{align}
 \bigprob{\tau_- > t} 
&= \bigprob{\tau_- > t, t \leqs \tau_- \wedge 
\tau_+ \wedge \tau_{\cB_\perp(h_\perp)}}
+ \bigprob{\tau_+ \wedge \tau_{\cB_\perp(h_\perp)} < t < \tau_-} \\
&\leqs \bigprob{\tilde\phi_0(t) \geqs -d_0 \;\forall t\in[t_0,t_0 + \tilde 
c\eps]} 
+ \bigprob{\tau_+ \wedge \tau_{\cB_\perp(h_\perp)} < t \wedge \tau_-}\;.
\end{align}
We have already shown that the first term on the right-hand side is 
exponentially small, and the second term can be controlled as in the 
preceding results.  
\end{proof}


\appendix

\section{Some useful inequalities in Sobolev spaces}
\label{annexes} 

Given $\psi\in L^2(\T)$, Sobolev's inequality states that given any $p\geqs2$, 
for any $s>\frac12 - \frac1p$, there exists a finite constant $\CSob(s,p)$ such 
that
\begin{equation}
\label{sobolev_ineq} 
\norm{\psi}_ {L^p}\leqs \CSob(s,p)\norm{\psi}_ {H^s}\;.
\end{equation} 
The following estimate on products in Sobolev spaces applies to the case 
$s>\frac12$. A concise proof can be found 
in~\cite[Th\'eor\`eme~7]{Bourdaud_calcul_symbolique}. 

\begin{lemma}[Products in Sobolev Spaces]
If $s>\frac{1}{2}$ then there is a bilinear application 
\begin{equation}
\begin{array}{rccll}
       H^s(\mathbb{T})\ \times \ H^s(\mathbb{T}) &\longrightarrow& \ 
H^s(\mathbb{T})\; \\
      \quad (\psi,\phi)  &\longmapsto& \ \psi\phi\;,
 \end{array}
\end{equation} 
which coincides with the pointwise product and satisfies the estimate
\begin{equation}
\label{sob_prod} 
\norm{\psi\phi}_{H^s(\mathbb{T})}\leqs 
C\norm{\psi}_{H^s(\mathbb{T})}\norm{\phi}_{H^s(\mathbb{T})}
\end{equation} 
for some finite constant $C = C(s)$. 
\end{lemma}

While the above result does not hold if $s\leqs\frac12$, we have the 
following consequence of Young's inequality, a proof of which can be found, for 
instance, in~\cite[Lemma~4.3]{BG12a}.

\begin{lemma}[Young-type inequality]
Let $r, s,t \in (0,\frac12)$ be such that $t<r+s-\frac{1}{2}$. Then there exists 
a finite constant $C=C(r,s,t)$ such that
\begin{equation}
\label{ineq_young}
\norm{\psi \ast \phi}_{H^t}\leqs C\norm{\psi}_{H^r}\norm{\phi}_{H^s}<\infty\;.
\end{equation} 
\end{lemma} 


\bibliographystyle{plain}
{\small \bibliography{SR_ref}}

\newpage
{\small \tableofcontents}

\vfill

\bigskip\bigskip\noindent
{\small
Institut Denis Poisson (IDP) \\ 
Universit\'e d'Orl\'eans, Universit\'e de Tours, CNRS -- UMR 7013 \\
B\^atiment de Math\'ematiques, B.P. 6759\\
45067~Orl\'eans Cedex 2, France \\
{\it E-mail addresses: }
{\tt nils.berglund@univ-orleans.fr}, 
{\tt rita.nader@univ-orleans.fr}

\end{document}